\documentclass{article}
\usepackage[utf8]{inputenc}

\usepackage{amsmath, amsthm, amsfonts,stmaryrd,amssymb}
\usepackage[top=1in, bottom=1.25in, left=1.25in, right=1.25in]{geometry}
\usepackage{mathrsfs}
\usepackage{dsfont}                  
\usepackage{mathtools}                   
\usepackage{booktabs}
\usepackage{tikz}
\usepackage{dirtytalk}
\usepackage{todonotes}
\usepackage{enumerate} 
\usepackage{hyperref}
\usepackage{theoremref}
\usepackage{verbatim}
\usepackage{bm}

\usepackage{nomencl}
\makenomenclature

\usepackage{pgfplots}
\pgfplotsset{compat=1.15}
\usepackage{mathrsfs}
\usetikzlibrary{arrows}
\newtheorem{theorem}{Theorem}[section]
\newtheorem{remark}[theorem]{Remark}
\newtheorem{lemma}[theorem]{Lemma}
\newtheorem{corollary}[theorem]{Corollary}

\newtheorem{proposition}[theorem]{Proposition}

\theoremstyle{definition}
\newtheorem{definition}[theorem]{Definition}
\def\al{\alpha} 
\def\be{\beta} 
 
\def\ga{\gamma} 
\def\de{\delta} 
\def\ep{\varepsilon} 
\def\ze{\zeta} 
 
\def\th{\theta}

\def\la{\lambda}

\def\ta{\tau} 
\def\ph{\varphi} 
 
\def\ps{\psi}

\def\La{\Lambda}

\def\L{\mathcal{L}}

\def\R{{\mathbb R}}

\def\N{{\mathbb N}}

\DeclareMathOperator*{\argmin}{arg\,min}

\newcommand{\dd}{\mathrm{d}}
\newcommand\Relax{\mathrm{relax}}
\newcommand\spt{\mathrm{supp}}
\newcommand\cost{\mathrm{cost}}
\newcommand\pigor{{\Pi_{g^\perp}(\mu, \nu)}}
\newcommand\PP{\mathscr{P}}
\newcommand\MM{\mathscr{M}}
\newcommand\Lc{{\mathscr{L}}_c}
\newcommand\eps{\varepsilon} 
\newcommand\tal{\widetilde{\alpha}}
\newcommand\tfi{\widetilde{\phi}}
\newcommand\tps{\widetilde{\psi}}

\newcommand\txi{\widetilde{\xi}}

\newcommand\hal{\widehat{\alpha}}
\newcommand\hfi{\widehat{\phi}}
\newcommand\hps{\widehat{\psi}}
\newcommand\hl{\widehat{\lambda}}
\newcommand\hb{\widehat{\beta}}
\newcommand\hxi{\widehat{\xi}}
\newcommand\onD{{\overline{\nu}}_D}
\newcommand\SMm{\mathrm{Softmin}^{\eps}_ {\mu}}
\newcommand\SMn{\mathrm{Softmin}^{\eps}_{\nu}}

\usepackage{xcolor}
\definecolor{darkblue}{rgb}{0.0,0.0,0.5}
\definecolor{darkred}{rgb}{0.5,0.0,0.0}

\newcommand{\tth}{\tilde{\theta}}
\DeclareMathOperator*{\supp}{supp}
\usepackage{algpseudocode, algorithm}

\title{Weak optimal transport with moment constraints: constraint qualification, dual attainment and entropic regularization}
\author{Guillaume Carlier, Hugo Malamut, Maxime Sylvestre}

\date{\today}
\begin{document}

\maketitle

\begin{abstract}
    We consider weak optimal problems (possibly entropically penalized) incorporating both soft and hard  (including the case of the martingale condition) moment constraints. Even in the special case of the martingale optimal transport problem,  existence of Lagrange multipliers corresponding to the martingale constraint is notoriously hard (and may fail unless some specific additional assumptions are made). We identify a condition of qualification of the hard moment constraints (which in the martingale case is implied by well-known conditions in the literature) under which general dual attainment results are established. We also analyze the convergence of entropically regularized schemes combined with penalization of the moment constraint and illustrate our theoretical findings by numerically solving in dimension one, the Brenier-Strassen problem of Gozlan and Juillet \cite{Gozlan-Juillet} and a family of problems which interpolates between monotone transport and the left-curtain martingale coupling of Beiglb\"{o}ck and Juillet \cite{BeigJuilletl}.
\end{abstract}




\section{Introduction}\label{sec-intro}
The optimal transport problem of Monge and Kantorovich, consists, given two Borel probability measures $\mu$ and $\nu$ on two Polish spaces $X$ and $Y$ and an lsc  function $\cost$: $X\times Y \to \R_+$, in finding an optimal plan or coupling between $\mu$ and $\nu$, i.e. an element of the set $\Pi(\mu, \nu)$ of Borel probability on $X\times Y$ having $\mu$ and $\nu$ as marginals that makes the total transport cost minimal, i.e.  which solves the linear optimization problem
\begin{equation}\label{mkot}
\inf_{\pi \in \Pi(\mu, \nu)} \int_{X\times Y} \cost(x, y) \dd \pi(x,y).
\end{equation}
Optimal transport (OT) has become a mature field of research that has tremendously developed in the last three decades and a detailed account of the theory and many of its applications can be found in the textbooks of  Villani \cite{Villanibook} and Santambrogio \cite{FSbook}. Many extensions of the classical Monge-Kantorovich OT problem, driven by various applications, have emerged among which, relevant to the present work, we wish to emphasize:
\begin{itemize}

\item \textbf{Entropic optimal transport}: this is the case where an additional relative entropic term is added to the transport cost in \eqref{mkot} typically with a small prefactor $\eps$, this problem has received a lot of attention in the last decade since it is at the heart of the efficient Sinkhorn algorithm as emphasized by Cuturi's influential paper \cite{Cut13}. Entropic optimal transport  is also intimately related to the Schr\"{o}dinger bridge problem \cite{Leo14}, we refer the reader to the lecture notes \cite{Nutz21} for a self-contained exposition of the topic. 

\item \textbf{Martingale optimal transport}: in these problems motivated by robust pricing in mathematical finance \cite{beiglböck2013modelindependent}, \cite{GHLT}, one has to face an extra constraint on the plan $\pi$, namely that it has to be the law of a martingale. By  classical theorems of Strassen \cite{strassen} and Cartier, Fell and Meyer \cite{CFM} the existence of  martingale couplings with marginals $\mu$ and $\nu$ is  equivalent to the requirement that $\mu$ and $\nu$ are in the convex order. For extensions to orders associated with more general cones and involved irreducible decompositions, we refer the reader to the recent work of Ciosmak \cite{ciosmak2024localisationconstrainedtransportsi}, \cite{ciosmak2024localisationconstrainedtransportsii}.  If the existence of an optimal martingale plan can easily be obtained by the direct method, and some characterizations are known in dimension one --see \cite{BeigJuilletl}, \cite{Shadowcouplings},  \cite{mathiasnutztouzi2017}, \cite{Beiglbck2019dualattainment1d}-- a general characterization by duality arguments is a notoriously delicate issue. The dual attainment problem i.e. the existence of a Lagrange multiplier for the martingale constraint is difficult especially in several dimensions and typically requires sophisticated arguments or extra assumptions, see \cite{GKL}, \cite{DeMarch}, \cite{backhoffveraguas2023bass}, \cite{ciosmak2024localisationconstrainedtransportsi}, \cite{ciosmak2024localisationconstrainedtransportsii}. It is worth pointing out that martingale optimal transport is a special instance of optimal transport with moment constraints i.e. extra constraints on the plan $\pi$ of the form
\begin{equation}\label{hardmc}
\int_{X\times Y} g(x,y) \dd \pi_x(y)=0, \mbox{ for $\mu$-a.e. $x$}
\end{equation}
where $\pi_x\in \PP(Y)$ represents the conditional distribution of $y$ given $x$ induced by the coupling $\pi$ (see below). The martingale constraint, of course, corresponds to the case $g(x,y)=y-x$, $x$ and $y$ in $\R^d$. For an extension of Kantorovich duality to such moment constraints, see \cite{Zaev}. Note also that the particular case where $g$ only depends on $y$ arises in the context of \emph{Vector Quantile Regression} --see \cite{CCG} and \cite{CCGBeyond}, which leads to optimal transport subject to conditional independence constraints. 


\item \textbf{Weak optimal transport (WOT)}: Motivated by functional inequalities and previous works by Talagrand \cite{Talagrand} and Marton \cite{Marton}, Gozlan, Roberto, Samson and Tetali \cite{GRST} introduced a rich new class of optimal transport problems, called weak optimal transport which rapidly stimulated an intense stream of research. We refer in particular the reader to \cite{ApplicationsofWOT} where several important applications are discussed. To describe WOT, let us recall that  by the disintegration Theorem (see \cite{DellacherieMeyer}) any plan $\pi\in \Pi(\mu, \nu)$ can be disintegrated with respect to its first marginal $\mu$ as follows: there exists a (unique up to a $\mu$ null set)  family   $(\pi_x)_{x\in Y}$ of (conditional) probabilities on $Y$, Borel, in the sense that $x\in X \mapsto \pi_x(B)$ is Borel for every Borel subset  $B$ of $Y$, and such that $\pi=\mu \otimes \pi_x$ in the sense that for every continuous and bounded function $\beta$ on $X\times Y$, one has
\[\int_{X\times  Y} \beta \dd \pi=\int_X \Big( \int_Y \beta(x,y) \dd \pi_x(y) \Big) \dd \mu(x).\]
Given a weak cost function  $c$ : $X\times \PP(Y) \to \R\cup \{+\infty\}$, the authors of \cite{GRST} considered
\[\inf_{\pi\in \Pi(\mu, \nu)} \int_X c(x, \pi_x) \dd \mu(x).\]
Of course, in the linear case where $c(x,p)=\int_Y \cost(x,y) \dd p(y)$, one recovers the Monge-Kantorovich problem \eqref{mkot}, but WOT theory encompasses many other interesting situations. In particular for $X=Y=\R^d$ and 
\[c(x, p)=\begin{cases} 0 \mbox{ if $\int_Y y \dd p(y)=x$} \\+ \infty \mbox{ otherwise } \end{cases}\]
we recover the martingale constraint. Another interesting case, arising in particular in the Brenier-Strassen framework of  \cite{Gozlan-Juillet}, is when 
\[c(x,p)=\theta \Big(\int_Y y \dd p(y)-x\Big)\]
which can be seen as a penalization of the martingale constraint. Such costs are often referred to as Marton costs. A remarkable aspect of WOT theory is that, under adequate conditions on the weak cost (in particular convexity in the measure variable), it can be addressed by convex duality tools which generalize  Kantorovich duality see \cite{GRST}, \cite{ABCWOT}, \cite{beiglbock2025fundamentaltheoremweakoptimal}. 

\end{itemize}

\smallskip

\textbf{Contributions} In the present paper, we are interested in WOT problems involving both hard and soft moment constraints. By hard constraints, we mean that constraints of the form \eqref{hardmc} will be imposed on the coupling $\pi$ for some function $g$. By soft we mean that the overall cost will  involve a Marton term of the form $\theta(\int_X f(x,y) \dd \pi_x(y))$ where $\theta$ is convex and has to be thought as a penalization of the moments involving the function $f$. To sum up, we aim at studying generalizations of \eqref{mkot} which consist in minimizing
\[ J(\pi)=\int_{X} c(x, \pi_x) \dd \mu(x) + \int_{X} \theta\Big(  \int_Y f(x,y) \dd \pi_x(y) \Big) \dd \mu(x)\]
among the couplings $\pi\in \Pi(\mu, \nu)$ which satisfy the hard moment constraints \eqref{hardmc}. We will also consider the entropic regularization of such problems. Under quite general assumptions (involving in particular the convexity and some regularity of the WOT term $c(x, \pi_x)$) we will establish dual attainment results. It is a classical fact from convex analysis \cite{EkelandTemam} that dual attainment can typically be deduced from some stability of the value of the primal with respect to perturbations of the data (e.g. the marginals) and adequate qualification conditions, such as the Slater condition in convex programming or the celebrated Attouch-Brezis qualification condition \cite{AttouchBrezis} in Banach spaces.  As in the martingale constraint case, the most delicate point is the existence of Lagrange multipliers for the hard moment constraint \eqref{hardmc}. We will identify a qualification of the constraint condition of Slater type which guarantees some form of coercivity in the dual and ultimately leads to uniform a priori bounds and the existence of dual solutions.  Dual attainment results are detailed in section \ref{sec-attain} in Theorems \ref{existoptpot} and \ref{existoptpoteps} corresponding respectively to the unregularized and regularized  situations. In the context of martingale optimal transport, we will demonstrate that our qualification of the constraint condition is implied by conditions of nondeneracy and irreducibility which are commonly used in the literature.  We will also prove  in Theorem \ref{thm:stab-irred-support} that irreducibility and nondegeneracy are equivalent to a $W_\infty$ stability criterion for $\mu$ and $\nu$ in the convex order. We believe these new connections might be useful for future research in the field.   We  also analyze convergence results in regimes where the entropic parameter vanishes and the penalty on moment constraint increases, our previous results concerning existence and boundedness of the dual variables will also be important in establishing convergence rates. Finally, we  will use a variant of the SISTA (see \cite{CDGSSista}) hybrid algorithm for numerically solving in dimension one, the Brenier-Strassen problem of Gozlan and Juillet \cite{Gozlan-Juillet} and a family of problems which interpolates between monotone transport and the martingale monotone coupling of Beiglb\"{o}ck and Juillet \cite{BeigJuilletl},\cite{NutzWieselEMOT}, \cite{CCGBeyond}



\smallskip

\textbf{Outline} The paper is organized as follows. In section \ref{sec-qual}, we introduce the concept of qualification of moment constraints and in particular emphasize its role in the important case of martingale optimal transport (MOT) and its intimate connections with nondegeneracy and irreducibility conditions which are commonly used in the MOT literature. In section \ref{sec-duality}, we introduce dual formulations and establish strong duality results under quite general assumptions. In section  \ref{sec-attain},  we demonstrate that the notion of qualification of the moment constraints identified in section \ref{sec-qual} is a key tool to establish dual attainment results i.e. in the first place existence of Lagrange multipliers for the moment constaints in natural functional spaces  for both unregularized and  entropically regularized situations. Section \ref{sec-convergence} establishes qualitative and quantitative convergence results when one lets the entropic parameter vanish, the penalization of the moment constraints explode or both at the time in some natural asymptotic regime.  Finally in section \ref{sec-sista}, we illustrate the practical implications of the previous  theoretical findings by implementing the SISTA hybrid algorithm on the Brenier-Strassen problem of Gozlan and Juillet \cite{Gozlan-Juillet} and a family of problems which interpolates between monotone transport and the martingale monotone coupling of Beiglb\"{o}ck and Juillet \cite{BeigJuilletl}.

\smallskip

{\textbf{Notations, standing assumptions}} Eventhough we believe some of our results can be generalized to Polish spaces, to avoid technicalities, we will throughout the paper always assume that $X$ and $Y$ are compact metric spaces (and in some specific situations e.g. when we consider MOT or some entropic regularization is present, we will even sometimes assume that $Y$, or even $X$ and $Y$ are compact subsets of $\R^d$; this will always be clearly indicated when needed). We denote by $\MM(X)$ (respectively $\MM_+(X)$, $\PP(X)$) the space of Borel measures on $X$ (respectively nonnegative Borel measures, respectively Borel probability measures on $X$). Riesz's representation Theorem will enable us to identify the topological dual of $C(X)$ (always endowed with the uniform norm) with $\MM(X)$.   The previous considerations will of course apply to $Y$ in place of $X$. We will also always assume that the function $g$ defining the \emph{hard} moment constraints and the function $f$ defining the soft moment constraints are continuous (hence uniformly continuous) on $X\times Y$, more precisely $g\in C(X\times Y, \R^N)$ and $f\in C(X\times Y, \R^M)$. If $(Z, \mathrm{dist})$ is a compact metric space, we will always endow $\PP(Z)$ with  the weak star topology and recall that the latter is metrized by the Wasserstein distance $W_q$ of order $q\in [1, \infty)$ defined, for every $P$ and $Q$ in $\PP(Z)$ by
\[W_q(P, Q):= \Big( \min_{\pi \in \Pi(P, Q)} \int_{Z\times Z}  \mathrm{dist}(z_1,z_2)^q \dd \pi(z_1,z_2)\Big)^{\frac{1}{q}}.\]
We will also use the $\infty$-Wasserstein distance:
\[W_{\infty}(P, Q):=\min_{\pi \in \Pi(P, Q)} \pi-\mathrm{Esssup} \;  \mathrm{dist}(\cdot, \cdot).\]
Recall that the relative entropy between two probability measures $P$ and $Q$ on  $Z$ is defined by
\[H(P \vert Q) :=\begin{cases} \int_Z  \log \Big(\frac{\dd P}{\dd Q}\Big) \dd P \mbox{ if $P \ll Q $}, \\+ \infty \mbox{ otherwise}\end{cases}\]
where $P\ll Q$ means that $P$ is absolutely continuous with respect to $Q$. The fact that $P$ and $Q$ are equivalent i.e.  $P\ll Q$ and  $Q\ll P$ will be denoted by $P \sim Q$. For $\phi \in C(Z)$, a modulus of continuity of $\phi$ is a function $\omega$ : $\R_+\to \R_+$ such that $\omega(t)\to 0$ as $t\to 0^+$ and  $\vert \phi(z_1)-\phi(z_2)\vert \leq \omega(\mathrm{dist}(z_1,z_2))$ for every $(z_1, z_2)\in Z^2$. We will often use the notation $\omega_\phi$ for a modulus of continuity of $\phi$ and if needed assume that it is nondecreasing and concave. If $K$ is a relatively compact subset of $C(Z)$, by the Arzel\`a-Ascoli Theorem, the elements of $K$ have a common modulus of continuity which we will denote by $\omega_K$. We will finally use the notation $T_\# m$ for the pushforward (or image measure) of a probability $m$ by a measurable map $T$.
\section{Moment constraints and qualification}\label{sec-qual}

Let $g:X\times Y\to \mathbb{R}^N$ be the moment map and let $\mu \in \PP(X),\nu \in \PP(Y)$ be the prescribed marginals. The goal of this section is to describe conditions on the moment constraint 
\begin{equation*}
    \int_{X \times Y} g(x,y) d \pi_x(y) = 0,
\end{equation*}
where $\pi \in \Pi(\mu,\nu)$, which will guarantee dual attainment. This question is key to develop a computational approach because the hard moment constraint can prevent dual attainment. When $g(x,y)=y-x$, $x$ and $y$ in $\R^d$, we recover martingale optimal transport. In that case, dual attainment results are only known for $d=1$ under strong assumptions \cite{Beiglbck2019dualattainment1d,mathiasnutztouzi2017}. In the general context of weak optimal transport, the recent work \cite{beiglbock2025fundamentaltheoremweakoptimal} establishes dual attainment for a broad class of weak optimal transport costs, however the framework of \cite{beiglbock2025fundamentaltheoremweakoptimal} does not accommodate the hard moment constraint. In fact, as in finite-dimensional convex programming, dual attainment doesn't always hold, a classical assumption to ensure it is that the constraints satisfy a Slater-type qualification condition. 

\begin{definition}\label{def:slater}
    The moment constraint $g$ is said to be \emph{qualified} with respect to $(\mu,\nu)$ if:
    \begin{itemize}
        \item for every $\pi \in \Pi(\mu,\nu)$ we have :
        \begin{equation*}
            \int_{X\times Y} g(x,y) \dd\pi(x,y) = 0,
        \end{equation*}
        \item there exists $\eta > 0$ such that for every $r\in L^\infty(\mu)$ satisfying $\int_X r \dd\mu = 0$ and $\Vert r \Vert_\infty \leq \eta$ there is $\pi \in \Pi(\mu,\nu)$ such that 
        \begin{equation*}
        \int_{Y} g(x,y) \dd\pi_x(y) = r(x)    
        \end{equation*}
        for $\mu$ almost every $x$.
    \end{itemize}
\end{definition}

In the following paragraphs we present the links between commonly made assumptions in the literature and this Slater-type qualification of constraints.

\subsection{Martingale optimal transport}

In this paragraph, $X$ and $Y$ denote the \emph{same} convex compact subset of $\R^d$ with nonempty interior in order for the  martingale constraint to make sense. As stated above, the martingale optimal transport \cite{beiglböck2013modelindependent} is the main example to have in mind when thinking about moment constraints. First, it is important to note that the set of couplings satisfying the moment constraint, i.e. the set of martingale couplings, can be empty. Strassen's theorem \cite{strassen} ensures that a martingale coupling exists between $\mu$ and $\nu$ if and only if 
\begin{equation*}
   \forall h \in \Gamma_{0}(X) \quad  \int_X h \dd\mu \leq \int_Y h \dd\nu
\end{equation*}
where $\Gamma_{0}(X)$ is the set of convex functions on $X$. Classically it is enough for $(\mu,\nu)$ to satisfy this inequality for $h \in \Gamma_{0,1}(X)$, where $\Gamma_{0,1}(X)$ is the set of convex and $1$-Lipschitz functions on $X$. If these inequalities are satisfied, we say that $\mu$ is dominated by $\nu$ in the convex order and denote it $\mu \leq_{\mathrm{cvx}}\nu$. However in order to derive dual attainment in dimension higher than $1$ it is not enough to assume that $\mu \leq_{\mathrm{cvx}}\nu$. In that case, another assumption called \emph{irreducibility} is made. It states that it is possible to transfer mass between any two points in the support of $\mu \otimes \nu$ using a martingale.

\begin{definition}[\cite{BackhoffVeraguas2023ExistenceOB}]\label{def:irreducibility}
    The pair of probability measures $(\mu,\nu)$ is said to be \emph{irreducible} if there is $\pi \in \Pi(\mu,\nu)$ such that $\int_Y y \dd\pi_x(y) = x$ and $\pi_x \sim \nu$ (i.e.  $\pi_x$ and $\nu$ have the same negligible sets) for $\mu$-almost every $x$.
\end{definition}

The definition of irreducibility is slightly different from the one presented in the introduction of \cite{BackhoffVeraguas2023ExistenceOB}. However, the equivalence with the definition presented here is established in Theorem D.1 in the appendix of \cite{BackhoffVeraguas2023ExistenceOB}. The existence of a martingale transport between $\mu$ and $\nu$ imposes a geometric constraint on their supports.

\begin{lemma}\label{lem:inclusion-supports}
    Assume that $\mu \leq_{\mathrm{cvx}} \nu$ then
    \begin{equation*}
        \mathrm{supp}(\mu) \subset C_\nu
    \end{equation*}
    where $C_\nu = \mathrm{co}(\mathrm{supp}(\nu))$ and $\mathrm{co}$ denotes the convex envelope.
\end{lemma}

\begin{proof}
Define $h:\mathbb{R}^d \to \mathbb{R}$ as the distance to $C_\nu$ then $h$ is convex and $1$-Lipschitz hence $\int_{X} h \dd \nu=0 \geq \int_X h \dd \mu$ so that  $ \mathrm{supp}(\mu) \subset C_\nu$.

\end{proof}

The following non-degeneracy condition is a strenghtening of the previous inclusion:

\begin{definition}[\cite{backhoffveraguas2024gradientflowbassfunctional}]\label{def:nondegeneracy}
    We say that the pair $(\mu,\nu)$ is \emph{non-degenerate} if 
    \begin{equation*}
        \mathrm{supp}(\mu) \subset \mathrm{int}(C_\nu).
    \end{equation*}
    
\end{definition}

The non-degeneracy assumption actually gives a quantitative estimate on the distance between the support of $\mu$ and the boundary of $C_\nu$ as is shown in the following lemma, the proof is easy and left to the reader.

\begin{lemma}\label{lemma:support-distance}
Let $F$ be a compact set and $C$ a bounded open convex set such that $F \subset C$, then there is $\delta > 0$ such that $\mathrm{co}(F) + B(0,\delta) \subset C$
\end{lemma}


The goal of this section is to give an equivalent characterization of these two assumptions which can be easily recast as a Slater type condition. The first step is to interpret these conditions in terms of stability of the property of convex ordering. In fact, we will show below that the conditions in Definitions \ref{def:irreducibility} and \ref{def:nondegeneracy} are equivalent to $W_\infty$ stability of the convex ordering.

\begin{definition}\label{def:stab}
    We say that the pair $(\mu,\nu)$ is $W_\infty$-\emph{stable} if there exists $\eta > 0$ such that for every $\bar{\mu} \in \mathcal{P}(X)$ such that $\int_X x \dd\bar{\mu}(x) = \int_X x \dd\mu(x)$ and $W_\infty(\bar{\mu},\mu) \leq \eta$ we have $\bar{\mu} \leq_{\mathrm{cvx}} \nu$.
\end{definition}

This property is in fact stronger than the moment qualification stated in Definition \ref{def:slater}. We state here the notion of qualification of the moment constraint when it is a martingale constraint.

\begin{definition}\label{def:slater-martingale}
    The martingale constraint is said to be \emph{qualified} with respect to $(\mu,\nu)$ if there exists $\eta > 0$ such that for every $r\in L^\infty(\mu)$ satisfying $\int_X r \dd\mu = 0$ and $\Vert r \Vert_\infty \leq \eta$ there is $\pi \in \Pi(\mu,\nu)$ such that 
        \begin{equation*}
        \int_X y \dd\pi_x = x + r(x)    
        \end{equation*}
        for $\mu$ almost every $x$.
\end{definition}

\begin{proposition}\label{prop:stable-slater}
    If the pair $(\mu,\nu)$ is $W_\infty$-stable then the martingale constraint is qualified.
\end{proposition}

\begin{proof}
First observe that for $\pi \in \Pi(\mu,\nu)$ we have $\int_{X\times Y} (y-x) \dd\pi = \int_Y y \dd\nu - \int_X x \dd\mu = 0$ because $\mu \leq_{\mathrm{cvx}} \nu$.

It remains to prove the second part of the qualification property pertaining to the conditional moments. Let $\eta > 0$ be given by the $W_\infty$-stability of $(\mu,\nu)$. Let $r \in L^\infty(\mu)$ be such that $\int_X r \dd\mu = 0$ and $\Vert r \Vert_\infty \leq \eta$. Define $\mu_r$ by $\mu_r = (\mathrm{id} + r)_\# \mu$.  Then $\mu_r$ satisfies $\int_X x \dd\mu_r = \int_X x \dd\mu$ and $W_\infty(\mu_r,\mu)\leq \eta$ thus there is $\gamma \in \Pi(\mu_r,\nu)$ which is a martingale coupling. Define $\pi$ as $\mu \otimes \gamma_{x+r(x)}$. Then $\pi \in \Pi(\mu,\nu)$ and 
\begin{equation*}
    \int_Y y\dd\pi_x(y) = \int_Y y \dd\gamma_{x+r(x)}(y) = x+r(x)
\end{equation*}
which is exactly $\int_Y (y-x) \dd\pi_x(y) = r(x)$ for $\mu$ almost every $x$. 
\end{proof}

\begin{theorem}\label{thm:stab-irred-support}
    Let $\mu,\nu$ be two probability measures such that $\mu \leq_{\mathrm{cvx}} \nu$. The following two assertions are equivalent :
    \begin{enumerate}
        \item[$(i)$] the pair $(\mu,\nu)$ is irreducible and non-degenerate,
        \item[$(ii)$] the pair $(\mu,\nu)$ is $W_\infty$-stable.
    \end{enumerate}

\end{theorem}

\begin{proof}
\textbf{Proof of $(ii)\implies (i)$.} In order to prove non-degeneracy we will show that $\mathrm{supp}(\mu)+B(0,\de)\subset \mathrm{co}(\mathrm{supp}(\nu))$. 
Let $a \in B(0,\delta)$ and define $\bar{\mu} = \frac{1}{2}\left((T_a)_\# \mu + (T_{-a})_\# \mu\right)$ (where $T_a(x)=x+a$ i.e. $T_a$ is the translation by $a$).  Then it is clear that $\int_X x \dd\bar{\mu}(x) = \int_X x \dd\mu$ and $W_\infty(\bar{\mu},\mu) \leq \Vert a \Vert \leq \delta$. Thus by $(ii)$ we have $\bar{\mu} \leq_{cvx} \nu$. Note that by construction $\mathrm{supp}(\mu) + a\subset\mathrm{supp}(\bar{\mu})$. 
Since $\bar{\mu}\leq_{\mathrm{cvx}} \nu$ Lemma \ref{lem:inclusion-supports} implies $\mathrm{supp}\bar{\mu}\subset \mathrm{co}(\mathrm{supp}(\nu))$. Thus $\mathrm{supp}(\mu) + a\subset \mathrm{co}(\mathrm{supp}(\nu))$ for every $a \in B(0,\delta)$. This implies that $\mathrm{supp}(\mu) + B(0,\delta) \subset \mathrm{co}(\mathrm{supp}(\nu))$. We have proven that the pair $(\mu,\nu)$ is nondegenerate.

To prove irreducibility, we will construct a martingale transport plan using the qualification of the moment constraint. Indeed, Proposition \ref{prop:stable-slater} ensures that the constraint is qualified. Let $\eta$ be given by Lemma \ref{lem:slater-Hfinite}, which ensures moment qualification using only transport plans with finite entropy. It means that, if $\| r\|_\infty \leq \eta$, then $r(x)$ can be represented as $\int (y-x)\dd \ga_x(y)$ for some $\ga \in \Pi(\mu,\nu)$ with $H(\ga \mid \mu \otimes \nu) < + \infty$. Since $X\times Y$ is compact the map $(x,y) \in X\times Y \mapsto y-x$ is bounded by a constant $C$. Thus there is $\gamma \in \Pi(\mu,\nu)$ such that $H(\gamma\mid \mu \otimes \nu) < + \infty$ and for $\mu$ almost every $x$ we have $\int_Y (y-x) \dd\gamma_x(y) = -\frac{\eta}{C} \int_Y (y-x) \dd\nu(y)$. Now set $\pi = \alpha\gamma + (1-\alpha)\mu \otimes \nu$ where $\alpha \in (0,1)$ is such that $\alpha^{-1} = 1+\frac{\eta}{C}$. We have $H(\pi \mid \mu \otimes \nu) < + \infty$ by convexity of the entropy, $\pi \in \Pi(\mu,\nu)$ by convexity of the set of transport plans and for $\mu$ almost every $x$ we have 
\begin{align*}
\int_Y (y-x) \dd\pi_x(y) &= \al\int_Y (y-x) \dd\gamma_x(y) + (1-\al)\int_Y (y-x) \dd\nu(y)\\
&= \left(-\al\frac{\eta}{C} + 1-\al\right)\int_Y (y-x) \dd\nu(y) = 0.
\end{align*}
This proves that $\pi$ is a martingale transport. It remains to show that $\pi_x$ is equivalent to $\nu$ for $\mu$ almost every $x$. Since $\int_X H(\pi_x\mid \nu) \dd\mu(x) = H(\pi\mid \mu \otimes \nu) < +\infty$ we have for $\mu$ almost every $x$ that the measure $\pi_x$ is absolutely continuous with respect to $\nu$. By construction $\pi_x = \alpha \gamma_x + (1-\alpha)\nu$ thus $\nu$ is absolutely continuous with respect to $\pi_x$ because $\alpha < 1$.

\textbf{Proof of $(i)\implies (ii)$.} Without loss of generality we can assume that $\int_X x \dd\mu = \int_Y x \dd\nu = 0$. In particular $0 \in X$.
Lemma \ref{lemma:support-distance} ensures the existence of $\delta_1 > 0$ such that $K = \mathrm{co}(\mathrm{supp}(\mu)) + \bar{B}(0,\delta_1) \subset \mathrm{int}(\mathrm{co}(\mathrm{supp}(\nu)))$. Note that we also have $0\in K$ because $\int_X x \dd\mu = 0$. Define the following subset of convex functions 
\begin{equation*}
    S = \left\{\phi \in \Gamma_{0,1}(X): \phi(0) = 0, \mathrm{diam}(\partial \phi(K)) \geq \frac{1}{2} \right\}.
\end{equation*}
Observe that $S$ is compact thanks to Arzelà-Ascoli theorem. Indeed, $S$ is precompact because all the functions in $S$ are $1$-Lipschitz and for every $x\in X$ the set $\{\phi(x)\mid \phi \in S\}$ is bounded because $\phi(0)=0$, $\phi$ is $1$-Lipschitz and $X$ is bounded. It remains to show that $S$ is closed. Take a sequence $\phi_n$ in $S$ which converges uniformly towards $\phi$. For $n\in \mathbb{N}$ there are $x_n,y_n \in K$ and  $p_n \in \partial \phi_n(x_n),q_n \in \partial \phi_n(y_n)$ such that $\Vert p_n - q_n \Vert \geq \frac{1}{2}$. Since $\phi_n$ is $1$-Lipschitz $p_n,q_n$ are bounded and thus we can extract a subsequence such that $p_n \to p$ and $q_n \to q$, moreover compactness of $K$ allows us to have $x_n \to x$ and $y_n \to y$ up to a subsequence. For $n \in \mathbb{N}$ we have for any $y \in X$
\begin{equation*}
    \langle p_n,y-x_n \rangle + \phi_n(x_n) \leq \phi_n(y).
\end{equation*}
Thus passing to the limit shows that $p \in \partial \phi(x)$ and similarly $q \in \partial \phi(y)$. Also by passing to the limit we have $\Vert p - q \Vert \geq \frac{1}{2}$ which ensures that $\mathrm{diam}(\partial \phi(K)) \geq \frac{1}{2}$. Thus $\phi \in S$ and $S$ is closed.

Let $\delta_2 = \inf_{\phi \in S} \int_Y \phi \dd\nu - \int_X \phi \dd\mu$. Note that $\delta_2 \geq 0$ because $\mu \leq_{cvx} \nu$. By compactness of $S$ there is $\phi \in S$ which achieves the infimum. Assume that $\delta_2 = 0$ then $\int_Y \phi \dd\nu -\int_X \phi \dd\mu = 0$. Let $\pi \in \Pi(\mu,\nu)$ be a martingale coupling between $\mu$ and $\nu$ such that, for $\mu$ almost every $x$, $\nu$ is absolutely continuous with respect to $\pi_x$ which is possible thanks to irreducibility (Definition \ref{def:irreducibility}). Then we have 
\begin{equation*}
    \int_{X\times Y} (\phi(y) - \phi(x) -  p(x)\cdot (y-x)) \dd\pi(x,y) = 0
\end{equation*}
where $p$ is a measurable selection in $\partial \phi$, because $\int_Y (y-x) \dd\pi_x = 0$ for $\mu$ almost every $x$. Thus there is $x$ such that 
\begin{equation*}
\int_Y (\phi(y) - \phi(x) - p(x)\cdot (y-x)) \dd\pi_x(y) = 0
\end{equation*}
which implies that $p(x) \in \partial \phi(y)$ for $\pi_x$ almost every $y$. Since $\nu$ is absolutely continuous with respect to $\pi_x$, we deduce that $p(x) \in \partial \phi(y)$ for $\nu$ almost every $y$. The upper hemicontinuity of the subdifferential grants that $p(x) \in \partial \phi(y)$ for every $y \in \mathrm{supp}(\nu)$. And by convexity of $\phi$ we deduce that $p(x) \in \partial \phi(y)$ for every $y \in \mathrm{co}(\mathrm{supp}(\nu))$. It is classical that we have $\partial \phi(y) = \{p(x)\}$ for every $y \in \mathrm{int}(\mathrm{co}(\mathrm{supp}(\nu)))$ because $\phi$ coincides with an affine function on the interior. In particular we obtain that $\partial \phi(y) = \{p(x)\}$ for every $y \in K$ which contradicts $\mathrm{diam}(\partial \phi(K))\geq \frac{1}{2}$. We have proven by contradiction that $\delta_2 > 0$.

Set $\delta = \delta_1 \wedge \delta_2 > 0$. Let $\bar{\mu} \in \mathcal{P}(X)$ be such that $W_\infty(\bar{\mu},\mu) \leq \delta$ and $\int_X x \dd\bar{\mu} = \int_X x \dd\mu$. Because $\delta \leq \delta_1$ we have $\mathrm{supp}(\bar{\mu}) \subset K$. Let $\phi \in \Gamma_{0,1}(X)$, we will show that $\int_X \phi \dd\bar{\mu} \leq \int_Y \phi \dd\nu$. We distinguish two cases depending on the diameter of the subdifferential of $\phi$ on $K$.

\emph{First case $\mathrm{diam}(\partial \phi(K)) = 0$.} Then $\phi$ is affine on $K$ and thus $\int_X \phi \dd\bar{\mu} = \int_X \phi \dd\mu$ because $\int_X x \dd\bar{\mu} = \int_X x \dd\mu$. Since $\mu\leq_{cvx} \nu$ we have $\int_X \phi \dd\mu \leq \int_Y \phi \dd\nu$ and we deduce that $\int_X \phi \dd\bar{\mu} \leq \int_Y \phi \dd\nu$.

\emph{Second case $\mathrm{diam}(\partial \phi(K)) > 0$.} Set $\Delta = \mathrm{diam}(\partial \phi(K))> 0$. Let $p_0 \in \partial \phi(0)$, denote $\bar{\phi}$ the function $\frac{1}{2\Delta}(\phi - \langle p_0,\cdot\rangle )$ and define $\psi$ by infimal convolution between $\bar{\phi}$ and the norm
\begin{equation*}
\psi = \left(\frac{1}{2\Delta}(\phi - \langle p_0,\cdot\rangle )\right) \square \vert \cdot \vert.
\end{equation*}
Note that, by construction,  $\bar{\phi}$ is nonnegative and $\bar{\phi}(0)=0$ so that $0 \in \partial \bar{\phi}(0)=0$. Observe that $\mathrm{diam}(\partial \bar{\phi}(K)) = \frac{1}{2}$ and $0 \in \partial \bar{\phi}(K)$ because $0 \in K$ since $\int_X x \dd\mu = 0$. This implies that $\partial \bar{\phi}(K) \subset \bar{B}(0,\frac{1}{2})$ and that $\bar{\phi}$ is $\frac{1}{2}$-Lipschitz on $K$.  Note that $0 \leq \psi \leq \bar{\phi}$ so that $\psi(0)=0$ hence $\psi \in \Gamma_{0,1}(X)$ because infimal convolution preserves convexity and the infimal convolution with the norm enforce $1$-Lipschitz regularity. Our aim now is to show that $\psi \in S$ so it remains to to show that $\mathrm{diam}(\partial \psi(K)) \geq \frac{1}{2}$. Let $x \in K$ and $p \in \partial \bar{\phi}(x)$ then 
 for every $z\in X$, we have
\begin{align*}
    \psi(z)&=\min_{y\in X}\{ \bar{\phi}(y)+ \vert z-y\vert\}
    \geq  \bar{\phi}(x)+\min_{y\in X}\{ p\cdot (y-x)+ \vert z-y\vert\}\\
    & =\bar{\phi}(x) + p\cdot (z-x)+ \min_{y\in X}\{ p\cdot (y-z)+ \vert z-y\vert\}\\
    & \geq \bar{\phi}(x) + p\cdot (z-x) \geq \psi(x) + p\cdot (z-x)
\end{align*}
where we used  that $\vert p \vert \leq \frac{1}{2}\leq 1$ and $\psi(x) \leq \bar{\phi}(x)$ in the last line. This shows  that $p\in \partial \psi(x)$ (and also that $\psi$ and $\bar {\phi}$ actually agree on $K$). 
Thus  $\partial \bar{\phi}(K) \subset \partial \psi(K)$ and  $\mathrm{diam}(\partial \psi(K)) \geq  \mathrm{diam}(\partial \bar{\phi}(K)) = \frac{1}{2}$, which concludes the proof of $\psi \in S$. We then have the following sequence of inequalities
\begin{equation*}
 \int_X \bar{\phi} \dd\bar{\mu} \leq \int_X \psi \dd\bar{\mu} \leq \int_X \psi \dd\mu + W_1(\bar{\mu},\mu) \leq \int_X \psi \dd\mu + \delta_2 \leq \int_Y \psi \dd\nu \leq \int_Y \bar{\phi} \dd\nu
\end{equation*}
where the first inequality holds because $\bar{\phi} = \psi$ on $K$ which contains the support of $\bar{\mu}$, the second inequality holds because $\psi$ is $1$-Lipschitz, the third inequality holds because $W_1(\bar{\mu},\mu) \leq W_\infty(\bar{\mu},\mu) \leq \delta_2$, the fourth inequality holds because $\psi \in S$ and $\delta_2$ is the infimum of $\int_Y h \dd\nu- \int_X h \dd\mu$ over $h \in S$ and the last inequality holds because $\bar{\phi} \geq \psi$ by construction. Finally, since $\bar{\mu},\mu$ and $\nu$ all have the same mass and mean we have $\int_X \phi \dd\bar{\mu} \leq \int_X \phi \dd\nu$. In both cases we have $\int_X \phi \dd\bar{\mu} \leq \int_Y \phi \dd\nu$ for any $\phi \in \Gamma_{0,1}(X)$ i.e. $\bar{\mu} \leq_{cvx} \nu$. This concludes the proof of $(i)\implies (ii)$.
\end{proof}

\subsection{Vector quantile regression}\label{par-vqr}

Another class of moment constraints is the one arising in vector quantile regression \cite{CCG}. This a special case where the moment constraint only depends on $y$ and we will replace, slightly abusively, $g(x,y)$ by $g(y)$. Note that it is necessary that $\int_Y g(y) \dd\nu(y) = 0$ for the moment constraint to be feasible. In this situation Definition \ref{def:slater} is the consequence of a more conventional condition which is the non degeneracy of the moment constraint.

\begin{proposition}
    If $\int_Y g(y) \dd\nu(y) = 0$ and $\int_Y g(y)\otimes g(y) \dd\nu(y)$ is invertible then
    the moment constraint $g$ is qualified with respect to $(\mu,\nu)$.
\end{proposition}

\begin{proof}
    First observe that the structure of the moment constraints ensures that $\int_{X\times Y} g \dd\pi = 0$ for every $\pi \in \Pi(\mu,\nu)$. Before proving the result for a general $g$ we will show the result when $g$ is the identity. Thus $\nu$ has mean zero and its variance covariance matrix is invertible. By Jensen's inequality $\delta_0 \leq_{\mathrm{cvx}} \nu$. We will show that the pair $(\delta_0,\nu)$ is irreducible and nondegenerate. For irreducibility observe that $\delta_0 \otimes \nu$ is a martingale transport between $\delta_0$ and $\nu$ which obviously satisfies the irreducibility condition. Since the variance covariance matrix of $\nu$ is invertible we have $0 \in \mathrm{int} C_\nu$ where $C_\nu$ is the convex hull of the support of $\nu$. Theorem \ref{thm:stab-irred-support} ensures that the pair $(\delta_0,\nu)$ is $W_\infty$ stable. Let $\eta > 0$ be given by the stability. Let $r \in L^\infty(\mu)$ be such that $\int_X r \dd\mu=0$ and $\Vert r \Vert_{L^\infty(\mu)} \leq \eta$. Then $\mathrm{supp}(r_\# \mu) \subset \bar{B}(0,\eta)$ thus $W_\infty(\delta_0,r_\# \mu) \leq \eta$. Moreover $\int_X x dr_\# \mu = 0$. Thus by $W_\infty$ stability $r_\# \mu \leq_{\mathrm{cvx}} \nu$. Let $\gamma \in \Pi(r_\# \mu,\nu)$ be a martingale coupling. Define $\pi = \mu \otimes \gamma_{r(x)}$, then $\pi \in \Pi(\mu,\nu)$ and 
    \begin{equation*}
        \int_Y y \dd\pi_x = \int_Y y \dd\gamma_{r(x)} = r(x).
    \end{equation*}
    This proves that the moment constraint is qualified, with an $\eta$ depending only on $\nu$.
    
    Now assume that $g$ is a continuous function defined over $Y$. The measure $\nu_g = g_{\#} \nu$ has mean zero and by hypothesis its variance covariance matrix is invertible. Let $\eta > 0$ be given by the first part of this proof for $\nu_g$ and the identity moment constraint. Let $r \in L^\infty(\mu)$ be such that $\int_X r \dd\mu = 0$ and $\Vert r \Vert_{L^\infty(\mu)} \leq \eta$. By the first part of the proof there is $\gamma \in \Pi(\mu,\nu_g)$ such that $\int_{\mathbb{R}^N} a \dd\gamma_x(a) = r(x)$. Denote by $\tau$ the transport plan $(g,\mathrm{id})_\# \nu$ between $\nu_g$ and $\nu$ and define $\pi$ as
    \begin{equation*}
        \int_{X\times Y} f(x,y) \dd\pi(x,y) = \int_X \int_{\mathbb{R}^N} \int_Y f(x,y) \dd\tau_a(y) \dd\gamma_x(a) \dd\mu(x), \; \forall f\in C(X\times Y).
    \end{equation*}
    First we show that $\pi \in \Pi(\mu,\nu)$. Obviously the first marginal of $\pi$ is $\mu$ by construction. For the second marginal take $f \in C(Y)$ then
    \begin{align*}
        \int_{X\times Y} f(y) \dd\pi(x,y) &= \int_{X\times \mathbb{R}^N} \int_Y f(y) \dd\tau_a(y) \dd\gamma(x,a)
        = \int_{\mathbb{R}^N} \int_Y f(y) \dd\tau_a(y) \dd\nu_g(a)\\
        &= \int_{\mathbb{R}^N \times Y} f(y) \dd\tau(a,y) = \int_Y f(y) \dd\nu(y).
    \end{align*}
    Observe that $\int_Y g(y) \dd\tau_a(y) = a$. Indeed let $f \in C(\mathbb{R}^N,\mathbb{R}^N)$ then
    \begin{equation*}
        \int_{\mathbb{R}^N \times Y} f(a) \cdot g(y) \dd\tau(a,y) = \int_Y f(g(y)) \cdot g(y) \dd\nu(y) =  \int_{\mathbb{R}^N} f(a) \cdot a \dd\nu_g(a) 
    \end{equation*}
   which proves that $\int_Y g(y) \dd\tau_a(y) = a$ for $\nu_g$ almost every $a$. Denote by $A$ the set of $a$ for which the equality doesn't hold. Then $0 = \nu_g(A) = \gamma(X\times A)$ which implies that for $\mu$ almost every $x$ the equality $\int_Y g(y) \dd\tau_a(y) = a$ holds $\gamma_x$ almost surely.
    We now compute the value of the conditional moment
    \begin{equation*}
        \int_Y g(y) \dd\pi_x(y) = \int_{\mathbb{R}^N} \int_Y g(y) \dd\tau_a(y) \dd\gamma_x(a) = \int_{\mathbb{R}^N} a \dd\gamma_x(a) = r(x)
    \end{equation*}
    because $\int_Y g(y) \dd\tau_a(y) = a$ for $\gamma_x$ almost every $a$ as shown above. 
\end{proof}

\subsection{Slater type condition and approximation with finite entropy}

In this paragraph, we present some technical results which will be useful in the sequel and illustrate the utility of our qualification condition. 

The following lemma shows that qualification implies that $\nu$ gives positive mass to super level sets of the components of $g$. More precisely, we have

\begin{lemma}\label{lemma:moment_on_a_set}
    Assume that $\mu$ is not a Dirac mass.
    If the moment constraint is qualified (Definition \ref{def:slater}) there is $\rho > 0$ such that for any $e\in S^{N-1}(0,1)$ and $x \in \mathrm{supp}(\mu)$ there is $A$ measurable such that $\nu(A) \geq \rho$ and $g(x,y)\cdot e  \geq \rho$ for every $y \in A$.
\end{lemma}
\begin{proof}
    Since $\mu$ is not a Dirac mass there is $\delta > 0$ and $\kappa>0$ such that $\mu(B(x,\kappa)) \leq 1-\delta$ for every $x \in \mathrm{supp}(\mu)$. Indeed, assume by contradiction that there is no such $\delta$ and $\kappa$. Then there is a bounded sequence $\kappa_n \to 0$, a sequence $\delta_n \to 0$ and a sequence of $x_n \in \mathrm{supp}(\mu)$ such that $\mu(B(x_n,\kappa_n)) \geq 1-\delta_n$. Up to a subsequence we can assume that $x_n$ converges to $x \in \mathrm{supp}(\mu)$ by compactness of the support. By outer regularity of $\mu$ we have that $\mu(\{x\})= 1$ which contradicts the fact that $\mu$ is not a Dirac. Set $0 < \tilde{\kappa} \leq \kappa$ such that $\omega_g(\tilde{\kappa}) \leq \frac{\eta\delta}{4}$ where $\omega_g$ is the modulus of continuity of $g$.
    Now for $e \in S^{N-1}(0,1)$ and $x_0 \in \mathrm{supp}(\mu)$ define $r:= \eta e (\mu(B(x_0,\kappa)^c)1_{B(x_0,\kappa)} - \mu(B(x_0,\kappa)) 1_{B(x_0,\kappa)^c})$ where $\eta > 0$ is given by the qualification of the moment constraint. Then $\int_X r \dd\mu = 0$ and $\Vert r \Vert_\infty \leq \eta$. By the qualification assumption there is $\pi \in \Pi(\mu,\nu)$ such that $\int_Y g(x,y) \dd\pi_x(y) = r(x)$ for $\mu$ almost every $x$. Thus for $x \in B(x_0,\kappa)$ we almost surely have
    \begin{equation*}
        \eta \delta \leq \int_Y  g(x,y)\cdot e  \dd\pi_x(y)  \leq \frac{\eta \delta}{2} + \Vert g \Vert_\infty \pi_x\left(\left\{y : g(x,y)\cdot e \geq \frac{\eta \delta}{2}\right\}\right).
    \end{equation*}
   By definition of the modulus of continuity for $x \in B(x_0,\tilde{\kappa})$, and since $\omega_g(\tilde{\kappa}) \leq \frac{\eta\delta}{4}$, we have
    \begin{equation*}
        \left\{y: g(x,y)\cdot e  \geq \frac{\eta \delta}{2}\right\} \subset \left\{y :  g(x_0,y)\cdot e \geq \frac{\eta \delta}{4}\right\}.
    \end{equation*}
    Thus since $ B(x_0,\tilde{\kappa}) \subset B(x_0,\kappa)$ integration grants
    \begin{equation*}
        \frac{\eta \delta}{2\Vert g \Vert_\infty} \mu(B(x_0,\tilde{\kappa})) \leq \nu\left(\left\{y:  g(x_0,y)\cdot e  \geq \frac{\eta \delta}{4}\right\}\right).
    \end{equation*}
    Finally note that $\mu(B(x_0,\tilde{\kappa}))$ is uniformly lower bounded for $x_0 \in \mathrm{supp}(\mu)$. Otherwise by compactness of the support there would be $x_0$ in the support such that $\mu(B(x_0,\tilde{\kappa})) = 0$ which contradicts the definition of the support.
\end{proof}

We now turn our attention to the approximation of plans by plans with restricted entropy. This is a key step in the proof of dual attainment for weak optimal transport problems with an additional entropy term. We will show that it is possible to qualify the moment constraint with transport plans of bounded entropy. In order to do so we present a finite entropy approximation of transport plans which is adapted to moment constraint. This approximation, which we call sliced approximation, is a modification of the block approximation introduced in \cite{CDPS17}.

\begin{definition}
\label{def:sliced_approx} Assume that $Y$ is a compact subset of $\R^d$.
     For $i = (i_1,\ldots,i_d) \in \mathbb{Z}^d$ define $Q_i = [i_1,i_1+1) \times \ldots \times [i_d,i_d+1)$. For $\delta > 0$ we write $Q_i^\delta = \delta \cdot Q_i$. Let $\pi \in \Pi(\mu,\nu)$ and $\de>0$. The $\delta$-sliced approximation of $\pi$ denoted $\pi^\delta$ is defined by its disintegration
    \begin{equation*}
       \pi^\de = \mu \otimes \pi^\de_x \quad \text{and}\quad \dd\pi^\delta_x = \sum_{i \in \mathbb{Z}^d,\nu(Q_i^\de)>0} 1_{Q_i^\delta} \frac{\pi_x(Q_i^\delta)}{\nu(Q_i^\delta)} \dd\nu.
    \end{equation*}
    Note that the conditional probability $\pi^\de_x$ is well defined because the support of $\nu$ is compact and thus the sum is finite.
\end{definition}
\begin{lemma}\label{lem:sliced-entropy} Assume that $Y$ is a compact subset of $\R^d$. For any $\delta \in (0,1)$ the $\de$-sliced approximation $\pi^\de$ of any $\pi \in \Pi(\mu,\nu)$ satisfies :
    \begin{itemize}
        \item $W_\infty(\pi,\pi^\de) \leq \sqrt{d}\delta$ as well as $W_\infty(\pi_x,\pi_x^\delta) \leq \sqrt{d}\delta$ for $\mu$ almost every $x$,
        \item $\pi^\delta \in \Pi(\mu,\nu)$,
        \item $H(\pi^\de \mid \mu \otimes \nu) \leq -d\ln(\de)+d\ln\left(\mathrm{diam}_\infty(Y)+2\right) $.
    \end{itemize}
\end{lemma}
\begin{proof}
We begin by showing that $W_\infty(\pi,\pi^\delta)\leq \sqrt{d}\delta$. Since $\pi \in \Pi(\mu,\nu)$ then $\mu$ almost surely for every $i \in \mathbb{Z}^d$ (because it is a countable set) if $\nu(Q_i^\delta)=0$ then $\pi_x(Q_i^\delta) = 0$ and thus $\pi_x^\delta(Q_i^\delta) = 0$ else if $\nu(Q_i^\delta) > 0$ by integration $\pi_x(Q_i^\delta) = \pi_x^\delta(Q_i^\delta)$. In any case we obtain that for $\mu$ almost every $x \in X$ we have $\pi_x(Q_i^\delta) = \pi_x^\delta(Q_i^\delta)$ this classically implies that $W_\infty(\pi_x,\pi_x^\delta) \leq \max_i \mathrm{diam}(Q_i^\delta) \leq \sqrt{d}\delta$. Since by construction $\pi^\delta$ and $\pi$ share the same first marginal $\mu$ we have
\begin{equation*}
    W_\infty(\pi,\pi^\delta) \leq \mathrm{ess sup}_\mu W_\infty(\pi_x,\pi_x^\delta) \leq \sqrt{d}\delta.
\end{equation*}

We will now show that $\pi^\de \in \Pi(\mu,\nu)$. By construction the first marginal of $\pi^\delta$ is $\mu$. Let $f$ be a continuous function on $Y$ we have
\begin{align*}
    \int_Y f(y) \dd\pi^\de(x,y) &= \sum_{i \in \mathbb{Z}^d,\nu(Q_i^\de)>0} \int_X  \int_{Q_i^\delta} f(y) \frac{\pi_x(Q_i^\delta)}{\nu(Q_i^\delta)} \dd\nu(y) \dd\mu(x)\\
    &=  \sum_{i \in \mathbb{Z}^d,\nu(Q_i^\de)>0} \frac{\pi(X\times Q_i^\delta)}{\nu(Q_i^\delta)}  \int_{Q_i^\delta} f(y) \dd\nu(y) \\
    &= \sum_{i \in \mathbb{Z}^d,\nu(Q_i^\de)>0}   \int_{Q_i^\delta} f(y) \dd\nu(y) = \int_Y f(y) \dd\nu(y)
\end{align*}
where the third equality holds because the second marginal of $\pi$ is $\nu$. Thus $\pi^\de \in \Pi(\mu,\nu)$.

For the upper bound on the entropy first oberve that $H(\pi^\de\mid \mu \otimes \nu ) = \int_X H(\pi^\de_x\mid \nu) \dd\mu(x)$. For $x\in X$ we have
    \begin{equation*}
        H(\pi^\de_x\mid \nu) = \sum_{i \in \mathbb{Z}^d,\nu(Q_i^\de)>0} \pi_x(Q_i^\delta) \ln\left(\frac{\pi_x(Q_i^\delta)}{\nu(Q_i^\delta)}\right) \leq - \sum_{i \in \mathbb{Z}^d,\nu(Q_i^\de)>0} \pi_x(Q_i^\delta) \ln\left(\nu(Q_i^\delta)\right),
    \end{equation*}
   because $\pi_x(Q_i) \in [0,1]$ and $t\ln(t)\leq 0$ for $t \in [0,1]$. Integrating over $x$ with respect to $\mu$ we obtain
    $H(\pi^\de\mid \mu \otimes \nu) \leq -\sum_{i \in \mathbb{Z}^d} \nu(Q_i)\ln\left(\nu(Q_i)\right)$. By convexity of $t\mapsto t\ln(t)$ we have 
    \begin{equation*}
        -\sum_{i \in \mathbb{Z}^d} \nu(Q_i)\ln\left(\nu(Q_i)\right) \leq \ln(N_\delta)
    \end{equation*}
    where $N_\delta$ is the number of $Q_i^\delta$ such that $\nu(Q_i^\delta) > 0$. By construction $\bigcup_{i \in \mathbb{Z}^d,\nu(Q_i^\delta)>0} Q_i^\delta \subset Y + \bar{B}_\infty(0,\delta)$, thus since the $Q_i^\delta$ are disjoint and have the same Lebesgue measure we have $N_\delta \delta^d \leq \vert Y + \bar{B}_\infty(0,\delta)\vert \leq \left(\mathrm{diam}_\infty(Y)+2\delta\right)^d$. Finally, we obtain
    \begin{equation*}
        H(\pi^\de\mid \mu \otimes \nu) \leq -d\ln(\delta) + d\ln\left(\mathrm{diam}_\infty(Y)+2\right).
    \end{equation*}
\end{proof}

This sliced approximation allows us to limit ourselves to plans with restricted entropy (not only finite but with an a priori bound), as stated precisely in the following lemma:
\begin{lemma} \label{lem:slater-Hfinite}
    Assume that $Y$ is a compact subset of $\R^d$ and that the moment constraint is qualified with $\eta > 0 $. Let $\delta \in (0,1)$ be such that $\omega_g (\sqrt{d}\delta)\leq \frac{\eta}{2}$. Then for any $r \in L^\infty$ satisfying $\int_X r \dd\mu =0, \Vert r \Vert_\infty \leq \eta/2$ there is $\pi \in \Pi(\mu,\nu)$ such that $H(\pi \mid \mu \otimes \nu) \leq -d \ln(\delta)+d\ln\left(\mathrm{diam}_\infty(Y)+2\right)$ and $\int_Y g(x,y)\dd\pi_x(y) = r(x)$.
\end{lemma}
\begin{proof}
Let $\eta>0$ be given by the qualification of the moment constraint. Let $\delta \in (0,1)$ be such that $\omega_g(\sqrt{d}\delta) \leq \frac{\eta}{2}$.
Assume by contradiction that there is $r \in L^\infty(\mu)$ such that $\int_X r \dd\mu = 0$, $\Vert r \Vert_\infty \leq \frac{\eta}{2}$ and there is no $\pi \in \Pi(\mu,\nu)$ satisfying $H(\pi \mid \mu \otimes\nu ) \leq -d\ln(\delta)+d\ln\left(\mathrm{diam}_\infty(Y)+2\right)$ such that $\int_Y g(x,y) \dd\pi_x(y) = r(x)$ for $\mu$ almost every $x$.  Define $K$ to be the set of conditional moments:
\begin{equation*}
    K = \left\{x \mapsto \int_Y g(x,y) \dd\pi_x(y) :  \pi \in \Pi(\mu,\nu), H(\pi\mid\mu\otimes \nu)\leq -d\ln(\delta)+d\ln\left(\mathrm{diam}_\infty(Y)+2\right)\right\}.
\end{equation*}
Note that $K$ is convex and it is the image of 
\begin{equation*}
    P = \left \{\pi \in \PP(X\times Y) \mid \pi \in \Pi(\mu,\nu), H(\pi\mid\mu\otimes \nu)\leq -d\ln(\delta)+d\ln\left(\mathrm{diam}_\infty(Y)+2\right) \right\}
\end{equation*}
under the map $M: \pi \mapsto (x \mapsto \int_Y g(x,y) \dd\pi_x(y))\in L^{\infty}((X, \mu), \R^N)$. It is easy to see that the map $M$ is weak-$\ast$/weak-$\ast$ sequentially continuous. The entropy is weak-$\ast$ lower semicontinuous and thus its sublevel set are weak-$\ast$ closed. The set $\Pi(\mu,\nu)$ is weak-$\ast$ sequentially compact. Since $\mu \otimes \nu \in P$, the set $P$ is non empty and it is weak-$\ast$ sequentially compact as the intersection of a sequentially compact and a sequentially closed set. Thus $K$ is weak-$\ast$ sequentially compact being the  image of a non empty weak-$\ast$ sequentially compact set under a sequentially continuous map. Note that $K$ is bounded because $g$ is bounded. Bounded sets in $L^\infty(\mu)$ are metrizable for the weak-$\ast$ topology because $L^1(\mu)$ is separable due to the compactness of $X$. Thus $K$ is weak-$\ast$ compact due to the metrizability of the weak-$\ast$ topology on $K$.

Since $K$ is convex and compact for the weak-$\ast$ topology Hahn-Banach's theorem ensures the existence of $l_0 \in L^1(\mu)$ which strictly separates $K$ from $\{ r\}$ i.e. $\int_X l_0(x) r(x) \dd\mu(x) > \sup_{k \in K} \int_X l_0(x) k(x) \dd\mu(x)$. 
Note that since all elements of $K$ and $r$ have zero mean $l_0$ can be taken to have a median equal to zero. Because zero is a median of $l_0$ Lemma \ref{lem:median} ensures the existence of $s \in L^\infty(\mu)$ such that $\int_X s \dd\mu = 0$, $\Vert s \Vert_\infty \leq 1$ and $\int_X  l_0\cdot s \dd\mu = \Vert l_0 \Vert_{L^1(\mu)}$. By the qualification of the moment constraint there is $\pi \in \Pi(\mu,\nu)$ such that for $\mu$ almost every $x$ we have $\int_Y g(x,y) \dd\pi_x(y) = \eta s(x)$.
Now take the $\delta$-sliced approximation $\pi^\delta$ of $\pi$. Then we have $\pi^\delta \in \Pi(\mu,\nu)$ and $H(\pi^\delta \mid \mu \otimes \nu) \leq -d\ln(\delta)+d\ln\left(\mathrm{diam}_\infty(Y)+2\right)$ by Lemma \ref{lem:sliced-entropy}. Moreover we have
\begin{align*}
   \int_X  l_0(x)\cdot r(x) \dd\mu(x) &> \int_{X\times Y} l_0(x)\cdot g(x,y) \dd\pi^\delta(x,y) \\
    &\geq \int_{X \times Y} l_0(x)\cdot g(x,y) \dd\pi(x,y)- \omega_g(\sqrt{d}\delta)\Vert l_0 \Vert_{L^1(\mu)} \\
    &\geq \eta \int_X l_0(x)\cdot s(x) \dd\mu(x) - \frac{\eta}{2}\Vert l_0 \Vert_{L^1(\mu)} \\
    &\geq \frac{\eta}{2} \Vert l_0 \Vert_{L^1(\mu)} \geq \int_X  l_0(x)\cdot r(x) \dd\mu(x)
\end{align*}
where the first inequality holds because $l_0$ strictly separates $K$ and $\{r\}$, the second inequality holds because $\mu$ almost surely $W_\infty(\pi_x,\pi_x^\delta) \leq\sqrt{d}\delta$ by Lemma \ref{lem:sliced-entropy}, the third inequality holds because $\int_Y g(x,y)\dd\pi_x(y)$ is equal to  $\eta s(x)$ for $\mu$ almost every $x$, the fourth inequality holds because $\int_X l_0(x)\cdot s(x) \dd\mu(x) = \Vert l_0 \Vert_{L^1(\mu)}$ and the last because  $\Vert r \Vert_\infty \leq \frac{\eta}{2}$. This proves the desired contradiction and ends the proof.

\end{proof}

\section{Duality}\label{sec-duality}

\subsection{Setting}

As explained in the introduction, our goal is to study weak optimal transport problems with both soft (given by the function $f$ and a penalty $\theta$) and hard (given by  the function $g$) moment constraints of the following form:
\begin{equation}
\label{pbwot}
\inf_{\pi \in \pigor} J(\pi):=\int_{X} c(x, \pi_x) \dd \mu(x) + \int_{X} \theta\Big(  \int_Y f(x,y) \dd \pi_x(y) \Big) \dd \mu(x)
\end{equation}
where the set $\pigor$ is the set of transport plans between $\mu$ and $\nu$ which satisfy the moment constraints given by the function $g$:
\[\pigor:=\left\{\pi \in \Pi(\mu, \nu) \;  : \; \int_X g(x,y) \dd \pi_x(y)=0 \mbox{ for $\mu$-a.e. $x\in X$}\right\}\]  
and $\pi_x$ denotes the disintegration of $\pi$ with respect to its first marginal $\mu$ i.e. $\pi=\mu \otimes \pi_x$. 
We  also consider the entropically penalized variant of \eqref{pbwot} with parameter $\eps>0$
\begin{equation}\label{pbwoteps}
\inf_{\pi \in \pigor} J_\eps(\pi):=J(\pi)+ \eps H(\pi \vert \mu \otimes \nu).
\end{equation}
We shall always assume that $X$ and $Y$ are compact metric spaces, that $f\in C(X\times Y, \R^M)$ and $g\in C(X\times Y,\R^N)$ and consider the following assumptions:
\begin{itemize}
\item
\textbf{(H1)} $\theta$ : $\R^M \to \R_+$ is convex (hence continuous and even locally Lipschitz),

\item \textbf{(H2)} the map $(x,p) \in X\times \PP(Y) \mapsto c(x,p) \in \R_+$ is Borel, and for every $x\in X$, $c(x,.)$ is convex lsc  and proper on $\PP(Y)$ (equipped with the weak star topology) and there exists $\widetilde{\pi}=\mu \otimes \widetilde{\pi}_x \in \pigor$ such that $\int_X c(x, \widetilde{\pi}_x) \dd \mu(x) <+\infty$ (so that the value of \eqref{pbwot} is finite). If $\eps>0$ and if we consider \eqref{pbwoteps} for $\eps>0$  we shall strengthen the previous condition by  the requirement that there exists $\widetilde{\pi}\in \pigor$ such that $\int_X c(x, \widetilde{\pi}_x) \dd \mu(x) + \eps H(\widetilde{\pi} \vert \mu \otimes \nu) <+\infty$.

\item \textbf{(H3)} there exists $K$ a convex compact subset of $C(Y)$ such that for every $(x,p) \in X\times \PP(Y)$ one has the dual representation
\begin{equation}\label{dualc}
c(x,p)= \sup_{\gamma \in K} \left\{\int_Y \gamma(y) \dd p(y) -c^*(x,\gamma)\right\}
\end{equation}
where 
\begin{equation}\label{defdecstar}
c^*(x,\gamma):=\sup_{q\in \PP(Y)} \left\{ \int_Y \gamma(y) \dd q(y)-c(x, q)\right\}
\end{equation}
\item \textbf{(H4)} $\int_{X \times Y} g \dd \pi=0$ for every $\pi \in \Pi(\mu, \nu)$. 
\end{itemize}

Before going further, let us comment these assumptions. Condition \textbf{(H3)} seems to be the most demanding one. Let us first observe that defining $c^*$ by \eqref{defdecstar}, it always satisfies: 
\[ \forall (x,\gamma_1, \gamma_2) \in X\times C(Y)\times C(Y), \gamma_1 \leq \gamma_2 \Rightarrow c^*(x, \gamma_1) \leq c^*(x, \gamma_2)\]
and for all $(x, \gamma, \lambda)\in X\times C(Y)\times \R$,
\[c^*(x, \gamma+\lambda)=c^*(x,  \gamma)+ \lambda, \; c^*(x, \gamma) \leq \Vert \gamma \Vert_{\infty}\]
which implies that $c^*(x,.)$ is $1$-Lipschitz with respect to the uniform norm:
\begin{equation}\label{cstarlip1}
\forall (x,\gamma_1, \gamma_2) \in X\times C(Y)\times C(Y), \; \vert c^*(x, \gamma_1)-c^*(x, \gamma_2) \vert \leq \Vert \gamma_1-\gamma_2 \Vert_{\infty}.
\end{equation}

An easy consequence of \textbf{(H2)-(H3)} is the following:

\begin{lemma}\label{cstarisnormal}
If the weak cost $(x,p)\in X \times \PP(y) \mapsto c(x,p)$ satisfies \textbf{(H2)-(H3)}, then $c(x,.)$ is  weakly star continuous on $\PP(Y)$ for every $x\in X$ and $c^*$ is Borel on $X\times K$.
\end{lemma}

\begin{proof}
For $x\in X$ and $(p,q)\in \PP(Y)^2$, \eqref{dualc} implies that 
\[ \vert c(x,p)-c(x, q) \vert \leq \sup_{\gamma \in K}  \Big\vert \int_{Y} \gamma \dd (p-q) \Big\vert\]
which together with the compactness of $K$  implies the weak star continuity of $c(x,.)$. Denoting by $(p_n)_n$ a weakly star dense sequence of $\PP(Y)$ we then have for every $(x,\gamma)\in X\times K$, 
\[c^*(x, \gamma)=\sup_n \max_{k=1, \ldots, n} \left(\int_Y \gamma \dd p_k-c(x, p_k)\right)\]
so that $c^*$ is Borel. 

\end{proof}

The fact that \eqref{dualc} holds for a fixed convex compact set $K$  in assumption \textbf{(H3)} is restrictive but is satisfied in the following typical applications:
\begin{itemize}
\item linear cost $c(x, p):=\int_Y \cost(x, y) \dd p(y)$ for some $\cost \in C(X\times Y)$. In this case $c^*(x, \gamma)=-\gamma^\cost(x)$ where $\gamma^\cost$ is the $\cost$-concave transform of $\gamma$ i.e.
\[\gamma^\cost(x)=\min_{y\in Y}\{\cost(x,y)-\gamma(y)\}\]
then $\max_{\gamma \in C(Y)} \{\int_Y \gamma \dd p+ \gamma^\cost(x)\} \leq \int_Y \cost(x,y) \dd p(y)$ and the maximum is attained for $\gamma=\cost(x,.)$ so that \eqref{dualc} holds for $K$ being 
 the convex hull of the set $\{\cost(x,.), \; x\in X\}$, which is compact by uniform continuity of the $\cost$ on $X\times Y$.

\item least transport cost: consider a third compact metric space $Z$ and $\phi$: $X\times Y\times Z\to \R_+$ Borel and such that $\phi(x,.,.)$ is continuous on $Y\times Z$ for every $x$, and has a modulus of continuity  $\omega$ which is uniform in $x$ (this is the case of course if $\phi\in C(X\times Y\times Z)$) and $x\in X\mapsto \sigma^x \in \PP(Z)$ a (Borel) family of probability measures parametrized by $x$. In this setting consider 
\[c(x,p):=\min_{ q\in \Pi(p, \sigma^x)} \int_{ Y\times Z} \phi(x,y, z) \dd q(y, z)\]
then it is easy to see that
\[c^*(x, \gamma)=-\int_{Z} \gamma^{\phi_x}(z) \dd \sigma^x(z), \forall \gamma\in C(Y)\]
where $\gamma^{\phi_x}$ is the $\phi_x$-concave transform of $\gamma$ i.e.
\[\gamma^{\phi_x}(z)=\min_{y\in Y}\{\phi(x,y,z)-\gamma(y)\}\]
and  it is well-known  (by Kantorovich duality and the so-called double concavification trick, see \cite{FSbook}, \cite{Villanibook}) that \eqref{dualc} holds with $K$ the set of continuous function on $Y$ which admit $\omega$ as modulus of continuity and vanish at a given point $y_0\in Y$, this set being compact by the Arzel\`a-Ascoli Theorem; we see that assumption \textbf{(H3)} holds.

\end{itemize}

Note that assumption \textbf{(H4)} holds in the martingale case where $g(x, y)=x-y$ and $\mu \leq_{\mathrm{cvx}} \nu$ so that $\mu$ and $\nu$ share the same barycenter which implies  \textbf{(H4)}. It also holds in case (which contains vector quantile regression as a special case) where $g(x,y)=g(y)$. More precisely, in that case \textbf{(H4)} is implied by \textbf{(H2)} that stipulates that $\pigor$ is nonempty.

\subsection{A Fenchel-Rockafellar argument}

From now on, we assume that assumptions \textbf{(H1)-(H2)-(H3)} are satisfied. As can easily be guessed, problem \eqref{pbwot} naturally arises as the dual of a problem involving: Lagrange multipliers (denoted $\phi$ and $\psi$) for the marginal constraints, a dual variable  (denoted $\lambda$) for the soft moment term involving $\theta$ and $f$, a Lagrange multiplier (denoted $\alpha$) for the hard moment constraint involving $g$ and a dual variable (denoted $\beta$) for the weak OT cost $c$ term. Let us define 
\begin{equation}
E:=C(X)\times C(Y)\times C(X, \R^M) \times C(X, \R^N) \times C(X\times Y)
\end{equation}
and define $\Lambda \in \Lc(E, C(X\times Y))$ by 
\begin{equation}\label{defdeLambda}
\Lambda \xi (x,y):=\beta(x,y)+ \lambda(x) \cdot f(x,y)+\alpha (x) \cdot g(x, y)-\phi(x)-\psi(y), \; \forall (x,y)\in X\times Y,
\end{equation}
where $\xi=(\phi, \psi, \lambda, \alpha, \beta)\in E$ and $\lambda \cdot f$ (respectively $\alpha  \cdot g$) denote usual scalar product in $\R^M$ (respectively $\R^N$). We shall also shorten notations by writing more concisely
\[(\phi \oplus \psi)(x,y):=\phi(x)+\psi(y), (\lambda \odot f)(x,y):= \lambda(x) \cdot f(x,y), \; (\alpha \odot g)(x,y):= \alpha(x) \cdot g(x,y).\]
We will also often use the notation $c_x$ for $c(x,.)$, $c_x^*$ for $c^*(x,.)$ and $\beta_x$ for $\beta(x,.)$.
For $\xi=(\phi, \psi, \lambda, \alpha, \beta)\in E$, define
\begin{equation}\label{defdeF}
F(\xi):=\begin{cases} \int_X [c_x^*(\beta_x)+ \theta^*(\lambda(x))] \dd \mu(x)-\int_X \phi \dd \mu-\int_Y \psi \dd \nu \mbox{ if $\beta_x \in K$ for every $x\in X$}\\ + \infty \mbox{ otherwise}\end{cases}
\end{equation}
where $\theta^*$ is the Legendre transform of $\theta$, $c_x^*$ is defined by \eqref{defdecstar} and  $K$ is the convex compact set from assumption \textbf{(H3)}. Also, for $h\in C(X\times Y)$ set
\begin{equation}\label{defdeG}
G(h)=\begin{cases} 0  \mbox{ if $h \geq 0$ on $X\times Y$} \\ + \infty \mbox{ otherwise} \end{cases}
\end{equation}
and, for $\eps>0$,
\begin{equation}\label{defdeGeps}
G_\eps(h)=\eps \int_{X\times Y} \exp\Big({-\frac{h(x,y) }{\eps}} \Big) \dd \mu (x) \dd \nu(y).
\end{equation}
Now we consider the convex minimization problems
\begin{equation}\label{dualpbwot}
\inf_{\xi \in E} \{F(\xi)+ G(\Lambda \xi)\}
\end{equation}
and for $\eps>0$,
\begin{equation}\label{dualpbwoteps}
\inf_{\xi \in E} \{F(\xi)+ G_\eps(\Lambda \xi)\}.
\end{equation}
Our assumptions \textbf{(H1)-(H2)-(H3)}  imply that problems \eqref{dualpbwot} and \eqref{dualpbwoteps} satisfy the assumptions of the Fenchel-Rockafellar Theorem (see \cite{EkelandTemam}) and since $C(X\times Y)^*\simeq \MM(X\times Y)$ by Riesz's Theorem, we have 
\begin{equation}\label{frpbwot}
\inf_{\xi \in E} \{F(\xi)+ G(\Lambda \xi)\}+ \min_{\pi \in \MM(X\times Y)} \{F^*(\Lambda^* \pi)+ G^*(-\pi)\} =0, 
\end{equation}
and for $\eps>0$
\begin{equation}\label{frpbwoteps}
 \inf_{\xi \in E} \{F(\xi)+ G_\eps(\Lambda \xi)\}+ \min_{\pi \in \MM(X\times Y)} \{F^*(\Lambda^* \pi)+ G_\eps^*(-\pi)\}=0.
\end{equation}
Obviously
\[G^*(-\pi)=\begin{cases} 0 \mbox{ if $\pi \in \MM_+(X\times Y)$} \\ + \infty \mbox{ otherwise}\end{cases}\]
and it is a classical and easy to check fact  (see \cite{Varadhan}) that
\begin{equation}\label{dualentropyformula}
G_\eps^*(-\pi)= \eps \begin{cases} \int_{X\times Y}  (\log(m)-1)m \dd( \mu \otimes \nu) \mbox{ if $\pi= m (\mu \otimes \nu)$ with $m\geq 0$} \\ + \infty \mbox{ otherwise}.\end{cases}
\end{equation}
 Let $\pi \in \MM_+(X\times Y)$, we have $F^*(\La^*\pi) = \sup_{\xi} \int \La\xi\dd\pi - F(\xi)$ so that, by definition of $F$ and $\Lambda$, 
\begin{align}
F^*(\Lambda^* \pi)&= \sup_{(\phi, \psi)\in C(X)\times C(Y)} \left\{-\int_{X\times Y} (\phi \oplus \psi) \dd \pi +\int_X \phi \dd \mu+ \int_Y \psi \dd \nu \right\}  \label{t1Fstar} \\
&+ \sup_{\alpha\in C(X, \R^N)} \int_{X\times Y} \alpha \odot g \dd \pi \label{t2Fstar}\\
&+ \sup_{\beta \in C(X,K)}   \left\{\int_{X\times Y} \beta \dd \pi-\int_X c_x^*(\beta_x) \dd \mu(x) \right\} \label{t3Fstar}\\
&+ \sup_{\lambda \in C(X, \R^M)}\left\{ \int_{X\times Y} \lambda \odot f \dd \pi -\int_{X} \theta^*(\lambda(x)) \dd \mu(x)  \right\}  \label{t4Fstar} 
\end{align}
where slightly abusing notations, by $C(X, K)$ we mean the set of functions $\beta\in C(X\times Y)$ such that $\beta_x \in K$ for every $x\in X$.
Observe that the sum of the right hand-side of \eqref{t1Fstar} and the expression in \eqref{t2Fstar} is $0$ if $\pi\in \pigor$ and $+\infty$ otherwise. We are  thus left to compute the expressions in \eqref{t3Fstar} and \eqref{t4Fstar} for $\pi=\mu \otimes \pi_x \in \Pi(\mu, \nu)$, which is the object of the next two Lemmas whose proof is postponed to the appendix:

\begin{lemma}\label{dualthetarep}
Assume \textbf{(H1)} and let  $\pi=\mu \otimes \pi_x \in \Pi(\mu, \nu)$, then there holds 
\begin{align}
\int_X  \theta\Big(  \int_Y f(x,y) \dd \pi_x(y) \Big) \dd \mu(x)= \sup_{\lambda \in C(X, \R^M)}\left\{ \int_{X\times Y} \lambda \odot f \dd \pi -\int_{X} \theta^*(\lambda(x)) \dd \mu(x)  \right\} \\
 =\sup_{\lambda \in C(X, \R^M), \Vert \lambda \Vert_{\infty }\leq C(\theta, f)}\left\{ \int_{X\times Y} \lambda \odot f \dd \pi -\int_{X} \theta^*(\lambda(x)) \dd \mu(x)  \right\} 
\end{align}
where in the last line $C(\theta, f)$ denotes the Lipschitz constant of $\theta$ on the ball of $\R^M$ centered at $0$ with radius $\Vert f \Vert_{\infty}+1$. 

\end{lemma}

\begin{lemma}\label{dualjeanpierre}
Assume \textbf{(H2)-(H3)} and let  $\pi=\mu \otimes \pi_x \in \Pi(\mu, \nu)$, then there holds 
\begin{equation}
\int_{X} c(x, \pi_x) \dd \mu(x)= \sup_{\beta \in C(X,K)} \left\{\int_{X\times Y} \beta \dd \pi-\int_X c_x^*(\beta_x) \dd \mu(x) \right\}.
\end{equation}
\end{lemma}

It therefore follows from Lemma \ref{dualthetarep} and Lemma  \ref{dualjeanpierre}  that for all $\pi \in \MM_+(X\times Y)$ we have
\begin{equation}\label{abamacron}
F^*(\Lambda^* \pi)=  \begin{cases} J(\pi) \mbox{ if $\pi \in \pigor$} \\ + \infty \mbox{ otherwise } \end{cases}.
\end{equation}
Using  \eqref{dualentropyformula},  for every $\pi \in \Pi(\mu, \nu)$ we have
 \begin{equation}\label{ababayrou}
 G_\eps^*(-\pi)= \eps H(\pi \vert \mu\otimes \nu)-\eps.
 \end{equation}
 We thus have:
\begin{theorem}\label{dualityformulasforwot}
Assume \textbf{(H1)-(H2)-(H3)}, then we have
\begin{equation}\label{equalpdvalues}
\min_{\pi\in \pigor} J(\pi)=\sup_{\xi \in E \; : \;  \Lambda \xi \geq 0} \{-F(\xi)\}
\end{equation}
and for $\eps>0$
\begin{equation}\label{equalpdvalueseps}
\min_{\pi\in \pigor} J_\eps(\pi)=\sup_{\xi \in E} \{-F(\xi)-\eps \int_{X\times Y} e^{-\frac{\Lambda \xi}{\eps}} \dd (\mu \otimes \nu)\}+\eps
\end{equation}
In particular this implies that both \eqref{pbwot} and \eqref{pbwoteps} admits solutions. Moreover adding the additional constraint $\Vert \lambda \Vert_{\infty} \leq C(\theta, f)$ where $C(\theta, f)$ is as in Lemma \ref{dualthetarep} in the minimization problems \eqref{dualpbwot} and \eqref{dualpbwoteps} does not affect their values.  If, in addition, \textbf{(H4)} holds one can further add the constraint
\begin{equation}\label{alphazeromean}
\int_{X} \alpha(x) \dd \mu(x)=0
\end{equation}
in the minimization problems \eqref{dualpbwot} and \eqref{dualpbwoteps} without affecting their values. 

\end{theorem}

\begin{proof}
Formulas \eqref{equalpdvalues} and \eqref{equalpdvalueseps} directly follow from \eqref{abamacron}, \eqref{ababayrou}, \eqref{frpbwot} and \eqref{frpbwoteps}.The fact that the constraint $\Vert \lambda \Vert_{\infty} \leq C(\theta, f)$ does not change the values of \eqref{dualpbwot} and \eqref{dualpbwoteps} directly follows from Lemma \ref{dualthetarep}. Finally, if \textbf{(H4)} holds, adding the constraint \eqref{alphazeromean} does not affect the value of the supremum in \eqref{t2Fstar} when $\pi\in \Pi(\mu, \nu)$ (which is automatic when the right hand side of \eqref{t1Fstar} is finite) which implies that adding constraint \eqref{alphazeromean} does not change the values of   \eqref{dualpbwot} and \eqref{dualpbwoteps}. 

\end{proof}

\subsection{Relaxation}

There is a priori no chance that the functionals in \eqref{dualpbwot} and \eqref{dualpbwoteps} admit minimizers among continuous functions of both variables, in particular in the variable $x$, since for instance $\alpha(x)$ plays a role of Lagrange multiplier for the hard moment constraint given by $g(x,.)$ and regularity of Lagrange multipliers with respect to a parameter is false in general. This suggests to enlarge the admissible dual set to functions which are typically continuous in $y$ but only measurable (and essentially bounded) in $x$. More precisely define 
\[E_{\Relax}:=L^{\infty}((X, \mu))\times C(Y)\times L^{\infty}((X, \mu), \R^M)\times L^{\infty}((X, \mu), \R^N)\times L^{\infty}((X,\mu), C(Y)).\]
For $\xi=(\phi, \psi, \lambda, \alpha, \beta)\in E_\Relax$ we can define $\Lambda \xi$ exactly as before i.e. 
\[\Lambda \xi :=\beta +  \lambda \odot f +\alpha \odot g -(\phi\oplus \psi) \in L^{\infty}((X, \mu), C(Y))\]
and $F(\xi)$ for $\xi\in E_{\Relax}$ by
\begin{equation}\label{defdeF}
F(\xi):=\begin{cases} \int_X [c_x^*(\beta_x)+ \theta^*(\lambda(x))] \dd \mu(x)-\int_X \phi \dd \mu-\int_Y \psi \dd \nu \mbox{ if $\beta_x \in K$ for $\mu$-every $x\in X$}\\ + \infty \mbox{ otherwise.}\end{cases}
\end{equation}
Setting for all $h \in  L^{\infty}((X, \mu), C(Y))$
\[G(h):=\begin{cases} 0 \mbox{ if $h \geq 0$} \\ + \infty \mbox{ otherwise} \end{cases} \mbox{ and } G_\eps(h):=\eps \int_{X\times Y} e^{-\frac{h}{\eps}} \dd (\mu \otimes \nu),\]
let us consider the relaxed versions of \eqref{dualpbwot} and \eqref{dualpbwoteps} for $\eps>0$
\begin{equation}\label{dualpbwotrel}
\inf_{\xi \in E_{\Relax}} \{F(\xi)+ G(\Lambda \xi)\}
\end{equation}
and
\begin{equation}\label{dualpbwotreleps}
\inf_{\xi \in E_{\Relax}} \{F(\xi)+ G_\eps(\Lambda \xi)\}.
\end{equation}
Since  $E \subset  E_{\Relax}$ we obviously have $\inf \eqref{dualpbwot} \geq \inf \eqref{dualpbwotrel}$ and $\inf \eqref{dualpbwoteps} \geq \inf \eqref{dualpbwotreleps}$. To prove the reverse inequality, we shall argue by a weak duality argument and conclude with Theorem \ref{dualityformulasforwot}. Let $\pi=\mu \otimes \pi_x\in \pigor$ and let $\xi\in E_{\relax}$, integrating Young's inequalities
\[c_x^*(\beta_x) +c_x(\pi_x) \geq \int_Y \beta_x \dd \pi_x, \; \mbox{ and } \theta^*(\lambda(x))+ \theta\Big(\int_Y f(x,y) \dd\pi_x(y)\Big) \geq \lambda (x) \cdot \int_Y f(x,y) \dd \pi_x(y)\]

and using the fact that $\Lambda \xi \in L^1(\pi)$ (in fact we even have $\Lambda \xi \in L^{\infty}(\pi)$) and $\int_{X\times Y} (\alpha \odot g)  \dd \pi=0$ and $\int_X \phi \dd \mu+ \int_Y \psi \dd \nu=\int_{X\times Y} (\phi\oplus \psi) \dd \pi$ because $\pi\in \pigor$, we arrive at 
\begin{align}
F(\xi) &\geq \int_{X\times Y} \beta \dd \pi-\int_{X} c_x(\pi_x) \dd \mu(x)  + \int_{X\times Y} (\lambda \odot f) \dd \pi-\int_{X}   \theta\Big(\int_Y f(x,y) \dd\pi_x(y)\Big) \dd\mu(x) \nonumber\\
&-\int_{X\times Y} (\phi\oplus \psi) \dd \pi \nonumber \\
& = \int_{X\times Y} \Lambda \xi \dd\pi -J(\pi). \label{dualineq}
\end{align}
So if $\Lambda \xi \geq 0$, since $\pi \geq 0$, we have $F(\xi)\geq -J(\pi)$
hence the weak duality inequality
\[\inf_{\xi \in E_{\Relax}} \{F(\xi)+ G(\Lambda \xi)\} \geq  \sup_{\pi \in \pigor} -J(\pi).\]
For $\eps>0$, we can use  Young's inequality 
\[G_\eps(\Lambda \xi)+ G_\eps^*(-\pi) \geq -\int_{X\times Y} \Lambda \xi \dd \pi\]
and \eqref{ababayrou} to obtain 
\[\inf_{\xi \in E_{\Relax}} \{F(\xi)+ G_\eps(\Lambda \xi)\} \geq \eps + \sup_{\pi \in \pigor} -J_\eps(\pi).\]

As a direct consequence of Theorem \ref{dualityformulasforwot}, we thus deduce

\begin{corollary}\label{relaxcoro}
Assume \textbf{(H1)-(H2)-(H3)}, then we have
\begin{equation}\label{equalpdvaluesrel}
\min_{\pi\in \pigor} J(\pi)=\sup_{\xi \in E \; : \;  \Lambda \xi \geq 0} \{-F(\xi)\}= \sup_{\xi \in E_{\Relax} \; : \;  \Lambda \xi \geq 0} \{-F(\xi)\}
\end{equation}
and for $\eps>0$
\begin{equation}\label{equalpdvaluesepsrel}
\min_{\pi\in \pigor} J_\eps(\pi)=\sup_{\xi \in E} \{-F(\xi)-G_\eps(\Lambda \xi)\}+\eps=\sup_{\xi \in E_{\Relax}} \{-F(\xi)-G_\eps(\Lambda \xi)\}+\eps.
\end{equation}

\end{corollary}

Of course we can again impose the extra constraint $\Vert \lambda \Vert_{\infty} \leq C(\theta, f)$ and (if \textbf{(H4)} is satisfied) that $\alpha$ has zero mean in the relaxed duals without changing their values.

\section{Dual attainment}\label{sec-attain}

Key to the existence of solutions to \eqref{dualpbwotrel} and \eqref{dualpbwotreleps} is the following Slater qualification condition whose importance (especially in the martingale constraint case) has been discussed in section \ref{sec-qual}. Let us consider:
\begin{itemize}
    \item \textbf{(H5)} There exists $\eta>0$ such that for every $r\in L^{\infty}((X,\mu), \R^N)$ such that
\[ \int_X r (x)\dd \mu(x)=0 \mbox{ and } \Vert r \Vert_{L^{\infty}} \leq \eta\]
there exists $\pi=\mu \otimes \pi_x$ such that
\[ \int_Y g(x,y) \dd \pi_x(y)=r(x), \; \mbox{ for $\mu$ a.e. $x\in X$}.\]
\end{itemize}
Note that $\textbf{(H4)-(H5)}$ is exactly the qualification of the moment constraint $g$ from Definition \ref{def:slater}.

\subsection{The unregularized case}

Our first main result is the existence of relaxed dual solutions:

\begin{theorem}\label{existoptpot}
If \textbf{(H1)-(H2)-(H3)-(H4)-(H5)} hold then  the relaxed dual problem \eqref{dualpbwotrel} admits at least one solution.

\end{theorem}

\begin{proof}
Let $\xi^n=(\phi^n, \psi^n, \lambda^n, \alpha^n, \beta^n)\in E^{\N}$ be a minimizing sequence for \eqref{dualpbwot} (hence also for  \eqref{dualpbwotrel} by Corollary \ref{relaxcoro}) i.e.
\begin{equation}\label{minseqxin}
\Lambda \xi^n \geq 0, \; \lim_n F(\xi^n)=\inf \eqref{dualpbwot}=\inf \eqref{dualpbwotrel} \in \R.
\end{equation}
Since we have assumed \textbf{(H4)}  and both the constraint $\Lambda \xi \geq 0$ and the functional $F$ are invariant by replacing $(\phi, \psi)$ by $(\phi-a, \psi+a)$ where $a\in \R$ is a constant we can further assume
\begin{equation}\label{normalizn}
\int_X \phi_n \dd\mu=0, \; \int_X \alpha_n \dd \mu=0, \; \forall n.
\end{equation}
Note that $\beta_x^n \in K$ for every $x\in X$ and every $n$ where $K$ is the compact subset of $C(Y)$ which appears in assumption \textbf{(H3)}. Denoting by $C$ a positive constant that may change from line to line, we may also assume, thanks to the compactness of $K$ in \textbf{(H3)} and the second part of Theorem \ref{dualityformulasforwot}:
 \begin{equation}\label{blbounded}
 \Vert \beta^n\Vert_{\infty} \leq C, \; \Vert \lambda^n\Vert_{\infty} \leq C. 
 \end{equation}
so that the constraint $\Lambda \xi^n \geq 0$ entails 
\begin{equation}\label{lxinpos}
\phi^n(x)+ \psi^n(y) \leq C + \alpha^n(x)\cdot  g(x, y), \; \forall (x,y) \in X\times Y.
\end{equation}
By \eqref{minseqxin}, assumption \textbf{(H2)} and using the fact that $\beta^n$ is uniformly bounded, we have
\begin{align*}
C &\geq F(\xi^n)= \int_{X} (c_x^*(\beta_x^n)+\theta^*(\lambda^n(x)))\dd \mu(x)-\int_{X} \phi^n \dd \mu-\int_Y \psi^n \dd \nu\\
&\geq \int_{X\times Y} \beta^n \dd \widetilde{\pi}  -\int_X c_x(\widetilde{\pi}_x) \dd \mu(x)  -\theta(0) -\int_{X} \phi_n \dd \mu-\int_Y \psi_n \dd \nu\\
&\geq -C -\int_{X} \phi^n \dd \mu-\int_Y \psi^n \dd \nu
\end{align*}
hence
\begin{equation}\label{macrontvc}
\int_{X} \phi^n \dd \mu+\int_Y \psi^n \dd \nu= \int_Y \psi^n \dd \nu \geq -C.
\end{equation}

\smallskip

\textbf{Step1: $L^1$ bound on $\alpha^n$ and consequences}

Let $\pi=\mu\otimes \pi_x \in \Pi(\mu, \nu)$, using \eqref{macrontvc} and integrating inequality \eqref{lxinpos} with respect to $\pi$ we get
\[\int_{X} \alpha^n(x) \Big(\int_Y g(x,y) \dd \pi_x(y) \Big) \dd \mu(x)\geq -C.\]
We deduce from Slater's condition \textbf{(H5)} that for every $r\in L^{\infty}((X,\mu), \R^N)$ with $\Vert r \Vert_{L^{\infty}} \leq \eta$ and $\int_X r \dd \mu=0$ we can find $\pi\in \Pi(\mu, \nu)$ such that $\int_Y g(x, y) \dd \pi_x(y)=-r(x)$ and then we obtain that 
\[\sup \left\{\ \int_X \alpha^n(x) \cdot r(x) \dd \mu(x): \;  \; r\in L^{\infty}((X, \mu), \R^N) , \; \Vert r \Vert_{L^{\infty}} \leq \eta, \; \int_X r \dd \mu=0\right\} \leq C\]
but it is  a well-known fact (see Lemma \ref{lem:median} in The Appendix for details) that the left-hand side of the inequality above coincides with $ \eta \min_{a\in \R^N} \Vert \alpha^n-a\Vert_{L^1(\mu)}$. This implies that  for some sequence of constants $a^n\in \R^N$, $\Vert \alpha^n-a^n\Vert_{L^1(\mu)} \leq C$. However note that the second condition in \eqref{normalizn} ensures $\vert a^n \vert = \left\vert \int (\alpha^n - a^n) \dd \mu \right\vert \leq \Vert \alpha^n-a^n\Vert_{L^1(\mu)}\leq C$ and therefore, we arrive at the bound
\begin{equation}\label{boundl1alpha}
\Vert \alpha^n\Vert_{L^1(\mu)} \leq C.
\end{equation}
Integrating inequality \eqref{lxinpos} with respect to $x$, we also obtain a universal upper bound on $\psi^n$:
\begin{equation}
\psi^n(y) \leq C + \int_X \alpha_n(x) \cdot g(x,y) \dd \mu(x) \leq C + \Vert \alpha^n\Vert_{L^1(\mu)} \;  \Vert g \Vert_{\infty}
\end{equation}
denoting by $\psi^n_+ :=\max(\psi^n , 0)$ the positive part of $\psi^n$,  we deduce from \eqref{macrontvc} 
\begin{equation}\label{boundl1psi}
\Vert \psi^n_+ \Vert_{\infty} \leq C, \; \Vert \psi^n \Vert_{L^1(\nu)} \leq C.
\end{equation}
Using these bounds and integrating inequality \eqref{lxinpos} with respect to $y$, yields $\phi_n(x) \leq C + \vert \alpha^n(x)\vert \;  \Vert g \Vert_{\infty}$ hence a uniform $L^1(\mu)$ bound on $\phi^n_+$ hence also on $\phi^n$ by the first condition in \eqref{normalizn}:
\begin{equation}\label{boundl1phi}
\Vert \phi^n\Vert_{L^1(\mu)} \leq C.
\end{equation}

\smallskip

\textbf{Step 2: alternate minimization, construction of a bounded in $L^{\infty}$ minimizing sequence}

\smallskip

Given $(\lambda, \alpha, \beta) \in C(X, \R^M) \times C(Y, \R^N)\times C(X\times Y)$, define the "shadow transport cost" 
\begin{equation}\label{ozzycost}
c_{\lambda, \alpha, \beta}:=\beta + \lambda \odot f+ \alpha \odot g.
\end{equation}
The partial minimization problem which consists given $(\lambda, \alpha, \beta)$ in 
\[\inf \{F(\phi, \psi, \lambda, \alpha, \beta) \; : \;  (\phi, \psi) \in C(X)\times C(Y), \; \Lambda(\phi, \psi, \lambda, \alpha, \beta)\geq 0\}\]
is obviously equivalent to the following 
\begin{equation}\label{kantodualshadow}
\sup \left\{\int_X \phi \dd \mu + \int_Y \psi \dd \nu\; : \; (\phi, \psi) \in C(X)\times C(Y)\; : \; \phi\oplus \psi \leq c_{\lambda, \alpha, \beta} \right\}
\end{equation}
 which has been very much investigated in optimal transport since it is the so-called  Kantorovich dual of the optimal transport problem between $\mu$ and $\nu$ for the cost $c_{\lambda, \alpha, \beta}$, it is in particular well-known (see \cite{FSbook}, \cite{Villanibook}) that this problem admits at least one solution $\phi$, $\psi$ consisting of $c_{\lambda, \alpha, \beta}$ conjugate potentials i.e. such that for every $(x,y)\in X\times Y$
 \begin{equation}\label{cconjugacy}
 \phi(x)=\min_{y\in Y}\{c_{\lambda, \alpha, \beta}(x,y)-\psi(y)\}, \; \psi(y)=\min_{x\in X} \{c_{\lambda, \alpha, \beta}(x,y)-\phi(x)\}.
 \end{equation}
again by the invariance $(\phi, \psi)\mapsto (\phi-a, \psi+a)$ we may also normalize $\phi$ in such a way that $\int_X \phi \dd \mu=0$. If we replace $\phi^n$ and $\psi^n$ by a solution of  \eqref{kantodualshadow}  for $(\lambda, \alpha, \beta)=(\lambda^n, \alpha^n, \beta^n)$, we still obtain a minimizing sequence for which we can further assume:
\begin{equation}\label{cconjugacyphin}
 \phi^n(x)=\min_{y\in Y}\{\beta^n(x,y) + \lambda^n(x)\cdot  f(x,y)+ \alpha^n(x) \cdot g(x,y)-\psi^n(y)\}, \; \forall x\in X
  \end{equation}
and
\begin{equation}\label{cconjugacypsin}
 \psi^n(y)=\min_{x\in X}\{ \beta^n(x,y) + \lambda^n(x)\cdot  f(x,y)+ \alpha^n(x) \cdot g(x,y)-\phi^n(x)\}, \; \forall y\in Y. 
  \end{equation}
Now define
\[\phi^n_0(x)=\min_{y\in Y}\{\beta^n(x,y) + \lambda^n(x)\cdot  f(x,y)-\psi^n(y)\}, \; \forall x\in X\]
and the (Borel but possibly discontinuous) competitor of $\alpha^n$
\[\tal^n(x):=\begin{cases} \alpha^n(x) \mbox{ if  $\phi^n(x) \geq \phi_0^n(x)$}\\ 0 \mbox{ otherwise}\end{cases}\]
and 
\begin{equation}\label{ctfin}
\tfi^n(x)=\min_{y\in Y}\{\beta^n(x,y) + \lambda^n(x)\cdot  f(x,y)+ \tal^n(x) \cdot g(x,y)-\psi^n(y)\}, \; \forall x\in X.
\end{equation}
By construction $(\tfi^n, \psi^n, \lambda^n, \tal^n, \beta^n)\in E_{\Relax}$, $\Lambda(\tfi^n, \psi^n, \lambda^n, \tal^n, \beta^n)\geq 0$ in $L^{\infty}((X,\mu), C(Y))$ and $\tfi^n=\max(\phi^n, \phi_0^n) \geq \phi^n$ so that $F(\tfi^n, \psi^n, \lambda^n, \tal^n, \beta^n) \leq F(\xi^n)$ and  $(\tfi^n, \psi^n, \lambda^n, \tal^n, \beta^n)$ is again a minimizing sequence for \eqref{dualpbwotrel} (but now in $E_\Relax$). From the fact that $\tfi^n(x) \geq \phi_0^n(x)$, by the definition of $\phi_0^n(x)$ and the uniform bounds \eqref{blbounded} on $\beta^n$ and $\lambda^n$ and  \eqref{boundl1psi} on $\psi^n_+$, we get:
\[\begin{split}
-C  \leq   \phi_0^n(x) \leq  \tfi^n(x) = \min_{y\in Y}\{\beta^n(x,y) + \lambda^n(x)\cdot  f(x,y)+ \tal^n(x) \cdot g(x,y)-\psi^n(y)\} \\
  \leq C + \inf_{y \in Y} \{ \tal^n(x) \cdot g(x,y)-\psi^n(y)\}.
  \end{split}\]
Assume for a moment that $\mu$ is not a Dirac mass, then, thanks to Lemma \ref{lemma:moment_on_a_set}, there exists $\rho>0$ such that for every $x\in \spt(\mu)$, we can find $A\subset Y$ such that $\nu(A)\geq \rho$ and $\tal^n(x) \cdot g(x,y) \leq -\rho \vert \tal^n(x)\vert$ on $A$ hence
\[-C \leq  \frac{1}{\nu(A)} \int_{A} (\tal^n(x) \cdot g(x,y)-\psi^n(y)) \dd \nu(y) \leq -\rho \vert \tal^n(x)\vert + \frac{1}{\rho} \Vert \psi^n\Vert_{L^1(\nu)}\]
together with the uniform $L^1$ bound on $\psi^n$ in \eqref{boundl1psi} we obtain a uniform bound on $\tal^n$:
\begin{equation}\label{talnbound}
\Vert \tal^n\Vert_{L^\infty(\mu)} \leq C. 
\end{equation}
In case $\mu=\delta_{x_0}$ is a Dirac mass, we can assume that $X=\{x_0\}$ and \eqref{normalizn} directly implies that $\alpha^n=\tal^n=0$ hence the bound \eqref{talnbound} is obvious in this case. Define now 
\begin{equation}\label{ctpsn}
\tps^n(y)=\mu-\mathrm{essinf}_{x\in X}\{ \beta^n(x,y) + \lambda^n(x)\cdot  f(x,y)+ \tal^n(x) \cdot g(x,y)-\tfi^n(x)\}, \; \forall y\in Y
\end{equation}
and $\txi^n:=(\tfi^n, \tps^n, \lambda^n, \tal^n, \beta^n)$. It directly follows from \eqref{ctfin} that $\tps^n \geq \psi^n$  and by construction $\Lambda \txi^n\geq 0$  in $L^{\infty}((X,\mu), C(Y))$ so $ \txi^n$ is admissible for the relaxed problem \eqref{dualpbwotrel} and $F(\txi^n)   \leq F(\tfi^n, \psi^n, \lambda^n, \tal^n, \beta^n) \leq F(\xi^n)$ so that $\txi^n$ is again a  minimizing sequence  for\eqref{dualpbwotrel}. Since $\beta^n$, $\lambda^n$, $\tal^n$ are uniformly bounded and $f$ and $g$ are uniformly continuous the family of functions $y\in Y \mapsto \beta^n(x,y) + \lambda^n(x)\cdot  f(x,y)+ \tal^n(x) \cdot g(x,y)$ admits a modulus of continuity that does not depend neither on $n$ nor on $x$, hence the sequence $\tps^n$ is uniformly equicontinuous, since it is also bounded in $L^1$, it is uniformly bounded and relatively compact in $C(Y)$ thanks to the Arzel\`a-Ascoli Theorem, therefore (up to a-not relabeled subsequence)
\begin{equation}
\Vert \tps^n \Vert_{\infty} \leq C, \; \mbox{ and there is some $\psi \in C(Y)$ such that } \Vert \tps^n-\psi\Vert_{\infty} \to 0.
 \end{equation}
 Let us observe that since $\tps^n \geq \psi^n$ by \eqref{ctfin} we have
\begin{equation}\label{conjugin}
\tfi^n(x)\geq \min_{y\in Y}\{\beta^n(x,y) + \lambda^n(x)\cdot  f(x,y)+ \tal^n(x) \cdot g(x,y)-\tps^n(y)\}, \; \forall x\in X
\end{equation}
but the definition of $\tps^n$ also yields that for $\mu$ a.e. $x$, one has 
\[\tfi^n(x) \leq \min_{y\in Y}\{\beta^n(x,y) + \lambda^n(x)\cdot  f(x,y)+ \tal^n(x) \cdot g(x,y)-\tps^n(y)\}\]
 hence \eqref{conjugin} is an equality $\mu$-a.e.. Having uniform bounds on $\tps^n$, $\lambda^n$, $\tal^n$ and $\beta^n$ gives an obvious uniform bound on the right-hand side of \eqref{conjugin} i.e. on $\tfi^n$. In conclusion, in this step we have found a minimizing sequence $\txi^n=(\tfi^n, \tps^n, \lambda^n, \tal^n, \beta^n)\in E_{\Relax}$ such that 
 \begin{equation}\label{whereweare}
  \beta^n_x \in K, \mbox{ for $\mu$ a.e $x$}, \; \Vert \tal^n\Vert_{L^\infty(\mu)}+  \Vert \tfi^n\Vert_{L^\infty(\mu)}+ \Vert \lambda^n\Vert_{L^\infty(\mu)}\leq C \mbox{ and $\tps^n \to \psi$ in $C(Y)$}.
\end{equation}

\smallskip

\textbf{Step 3: convergence and conclusion} Passing to subsequences if necessary, by Banach-Alaoglu theorem we may assume that there is some $(\phi, \alpha, \lambda) \in L^{2}((X,\mu), \R\times\R^N \times \R^M)$ to which $(\tfi^n, \tal^n, \lambda^n)$ converges weakly in $L^2(\mu)$. Of course, the bounds from \eqref{whereweare} pass to weak limits so that $(\phi, \alpha, \lambda) \in L^{\infty}((X,\mu), \R\times\R^N \times \R^M) $ and  $\Vert \alpha\Vert_{L^\infty(\mu)}+  \Vert \phi \Vert_{L^\infty(\mu)}+ \Vert \lambda\Vert_{L^\infty(\mu)}\leq C$. By compactness of $Y$, we can find $D:=\{y_k, \; k\in \N\}$ a dense sequence in $Y$ and define $\onD:= \sum_{k=1}^{\infty} 2^{-k} \delta_{y_k}$; and we can assume that $\beta^n$ converges weakly to some $\beta_D $ in $L^2(\mu \otimes \onD)$. By Mazur's Lemma (see \cite{EkelandTemam}, Chapter I), and recalling the convergence of $\tps^n$ to $\psi$ in $C(Y)$ (see \eqref{whereweare}), we can thus find for each $n$, an element $\hxi^n =(\hfi^n, \hps^n,  \hl^n, \hb^n, \hal^n)$ in the convex hull of $\{\txi^k, k\geq n\}$ such that
\begin{equation}\label{oyep1}
\hxi^n \to \xi:=(\phi, \psi, \lambda, \alpha, \beta_D) \mbox{ strongly in } L^2(\mu)\times C(Y)\times L^2(\mu, \R^M)\times L^2(\mu, \R^N)  \times L^2(\mu \otimes \onD).
\end{equation}
By convexity of $F$ and linearity of $\Lambda$ observe that $\hxi^n$  is again a  minimizing sequence  for \eqref{dualpbwotrel}. Taking a further extraction if necessary, we can also assume that 
\begin{equation}\label{oyep2}
(\hfi^n,  \hl^n, \hal^n) \to (\phi, \lambda, \alpha)  \mbox{ $\mu$-a.e.}
\end{equation}
and also  using \eqref{whereweare} and Lebesgue's dominated convergence Theorem that 
\begin{equation}\label{oyep3}
(\hfi^n,  \hl^n, \hal^n) \to (\phi, \lambda, \alpha)  \mbox{ in $L^p(\mu)$ for every $p\in [1,+\infty)$ and $\hps^n \to \psi$ in $C(Y)$}.
\end{equation}

 By the same argument, we may   assume that there exists a subset $A$ of $X$ such that $\mu(X\setminus A)=0$ and $\hb^n(x,y)\to \beta_D(x,y)$ for every $x\in A$ and every $y$ in the dense set $D$. For $x\in A$, $\hb^n(x,.)\in K$ hence by compactness of $K$ it has a subsequence which converges in $C(Y)$ to some limit which we denote  $\beta_x$, obviously $\beta_x$ and $\beta_D(x,.)$ agree on $D$ which since $D$ is dense and $\hb^n(x,.)$ is uniformly equicontinuous determines the limit $\beta_x$ uniquely and shows  the full convergence of $\hb_x^n$ to $\beta_x$. Setting $\beta(x,y):=\beta_x(y)$ we therefore have $\hb^n_x \to \beta(x,.)$ in $C(Y)$ for every $x\in A$. By the fact that $\sup_{x\in X} \Vert \hb^n_x \Vert_{\infty}\leq C$, Lebesgue's dominated convergence Theorem entails the following $L^p(\mu, C(Y))$ convergence
\begin{equation}\label{oyep4}
\Vert \hb^n-\beta\Vert^p_{L^p((X,\mu), C(Y))}= \int_X \Vert \hb^n_x  -\beta_x \Vert_{\infty}^p \dd \mu(x) \to 0, \; \forall p\in [1, +\infty).
\end{equation}
Denoting by $A'\subset X$ a set of full $\mu$-measure on which $(\hfi^n, \hl^n, \hal^n)$ converges to $ (\phi, \lambda, \alpha)$, we have convergence of $\Lambda \hxi^n$ to  $\Lambda \xi$ pointwise on $(A\cap A')\times Y$ as well as in $L^p(((X,\mu), C(Y))$ for every $p\in [1, \infty)$. Therefore  $\xi \in E_{\Relax}$, $\Lambda \xi \in L^{\infty} ((X,\mu), C(Y))$ and $\Lambda \xi \geq 0$ so that $\xi$ is admissible for  \eqref{dualpbwotrel}. By  \eqref{oyep1},  we firstly have
\[\lim_n \Big(\int_X \hfi^n \dd \mu+ \int_Y \ \hps^n \dd \nu \Big)= \int_X \phi \dd \mu+ \int_Y \psi \dd \nu,\]
secondly, by \eqref{oyep2} and Fatou's Lemma, we have
\[\liminf_n \int_X \theta(\hl^n(x)) \dd \mu(x) \geq  \int_X \theta(\lambda(x)) \dd \mu(x)\]
finally, it follows from \eqref{cstarlip1} that $c_x^*(\hb^n_x) \geq c_x^*(\beta_x)-\Vert \hb_x^n-\beta_x\Vert_{\infty}$, integrating this inequality and using \eqref{oyep4} yields
\[\liminf_n \int_X  c_x^*(\hb^n_x) \dd \mu(x) \geq  \int_X c_x^*(\beta_x)  \dd \mu(x).\]
All this shows that $F(\xi) \leq \liminf_n F(\hxi^n)$ so that $\xi$ solves  \eqref{dualpbwotrel}.

\end{proof}

\begin{remark}
Thanks to the results of paragraph \ref{par-vqr}, one can see that Theorem \ref{existoptpot} implies the existence of bounded Lagrange multipliers for optimal transport problems with conditional independence constraints arising in vector quantile regression under significantly weaker assumptions than  in \cite{CCGBeyond}. 
\end{remark}

\subsection{The entropic case $\eps>0$}

We now consider the case $\eps>0$ for which, assuming $Y$ is finite dimensional (so as to be able to apply Lemma \ref{lem:slater-Hfinite}), we have:

\begin{theorem}\label{existoptpoteps}
If $Y$ is a compact subset of $\R^d$ and \textbf{(H1)-(H2)-(H3)-(H4)-(H5)} hold then  the relaxed dual problem \eqref{dualpbwotreleps} admits at least one solution.
\end{theorem}

\begin{proof}
As in the proof of Theorem \ref{existoptpot} we pick $\xi^n=(\phi^n, \psi^n, \lambda^n, \alpha^n, \beta^n)\in E^{\N}$ a minimizing sequence for \eqref{dualpbwotreleps} for which the bounds \eqref{blbounded} and the normalization \eqref{normalizn} hold. We will again always denote by $C$ a positive constant that may vary from a line to another. Let $\pi\in \Pi(\mu, \nu)$, it follows from \eqref{ababayrou} and Young's inequality that 
\[G_\eps( \Lambda \xi^n) \geq -\int_{X\times Y} \Lambda \xi^n \dd \pi-\eps H(\pi \vert \mu \otimes \nu)+\eps\]
Since $\xi^n$ is a minimizing sequence, and bounding from below the integral terms involving $c_x^*(\beta_x^n)$ and $\theta^*(\lambda^n)$ as we did in the proof of Theorem \ref{existoptpot}, we arrive at 
\begin{align*}
C&\geq  -\int_{X} \phi^n \dd \mu -\int_Y \psi^n \dd \mu+ G_\eps (\Lambda \xi^n) \geq-\int_{X} \phi^n \dd \mu -\int_Y \psi^n \dd \mu -\int_{X\times Y} \Lambda \xi^n \dd \pi -\eps H(\pi \vert \mu \otimes \nu)+\eps\\
&=-C-\int_{X\times Y} (\Lambda \xi^n +\phi^n  \oplus \psi^n) \dd \pi -\eps H(\pi \vert \mu \otimes \nu)+\eps\\
&\geq -C -\int_{X \times Y} (\beta^n+\alpha^n \odot f + \lambda^n \odot g) \dd \pi  -\eps H(\pi \vert \mu \otimes \nu)\\
&\geq -C -\int_{X \times Y} \alpha^n \odot g \dd \pi  -\eps H(\pi \vert \mu \otimes \nu)
\end{align*}
where we used the fact $\pi \in \Pi(\mu, \nu)$ in the second line,  the  definition of $\Lambda \xi^n$ in the third one and the bound \eqref{blbounded} in the last one.  Note that the first inequality above together with the positivity of $G_\eps$ immediately gives the bound  \eqref{macrontvc}, and from the last one,  we get for every $k \geq 0$:
\[ \sup \left\{ -\int_{X \times Y} \alpha^n \odot g \dd \pi \; :  \pi \in \Pi(\mu, \nu) \; : \; H(\pi \vert \mu\otimes \nu) \leq k  \right\} \leq C+k \eps.\]
Let $\eta>0$ be as in assumption \textbf{(H5)}. Setting $k:= -d \log(\delta)+d\log\left(\mathrm{diam}_\infty(Y)+2\right)$, with $\de>0 $ small enough for having $\omega_g( \sqrt{d} \de) \leq \frac{\eta}{2}$ and $k \geq 0$,  we deduce from Lemma \ref{lem:slater-Hfinite} that the left-hand side of the previous inequality is larger than
\[\sup \left\{\ \int_X \alpha^n(x) \cdot r(x) \dd \mu(x), \;  \; r\in L^{\infty}((X, \mu), \R^N) , \; \Vert r \Vert_{L^{\infty}} \leq \frac{\eta}{2}, \; \int_X r \dd \mu=0\right\} \]
Thanks to \eqref{normalizn}, we derive a uniform bound \eqref{boundl1alpha} on $\Vert \alpha^n\Vert_{L^1}$ as before.

\smallskip

 Let us introduce the following notation for the so-called softmin of functions $a$ and $b$ in $L^{\infty}(\mu)$ and $L^{\infty}(\nu)$ respectively:
\[\SMm (a ):= -\eps \log \Big(\int_X e^{-\frac{a}{\eps}} \dd \mu\Big), \; \SMn( b):= -\eps \log \Big(\int_Y e^{-\frac{b}{\eps}} \dd \mu\Big)\]
and observe the elementary inequalities
\begin{equation}\label{basicsm}
 \mu-\mathrm{essinf} a \leq \SMm (a)  \leq \int_X a \dd \mu, \mbox{ and }  \; \nu-\mathrm{essinf} b \leq \SMn( b)  \leq \int_Y b \dd \nu. 
\end{equation}

Now observe that for fixed ($\lambda, \alpha, \beta)\in C(X, \R^{M})\times C(X, \R^N)\times C(X\times Y)$ (or more generally $L^{\infty}(\mu, \R^{M})\times L^{\infty}(\mu, \R^{N})\times L^{\infty}(\mu\otimes \nu)$) defining $c_{\lambda, \alpha, \beta}$ by \eqref{ozzycost}, the problem of minimizing $F+ G_\eps \circ \Lambda$ with respect to $(\phi, \psi)\in C(X)\times C(Y)$   is equivalent to following regularized version of \eqref{kantodualshadow}:
\begin{equation}\label{kantodualshadoweps}
\sup_{ (\phi, \psi) \in C(X)\times C(Y)} \left\{\int_X \phi \dd \mu + \int_Y \psi \dd \nu -\eps \int_{X\times Y} \exp\Big(-\frac{c_{\lambda, \alpha, \beta}+ \phi\oplus \psi}{\eps}  \Big) \dd (\mu \otimes \nu)  \right\}
\end{equation}
this type of minimization problems has been very much studied in the context of the Sinkhorn algorithm and in particular it is  well-known (see \cite{DMG}, \cite{CarlierSinkhorn}, \cite{Nutz2021} and the references therein) that it possesses a solution $(\phi, \ps)$  (where the normalization  $\int_X \phi \dd \mu=0$ can be imposed without loss of generality) which satisfies the following system (often called the Schr\"{o}dinger system, see \cite{Leo14}) of optimality conditions
\[\phi(x)=\SMn (c_{\lambda, \alpha, \beta} (x, \cdot) -\psi), \; \forall x\in X, \; \psi(y)=\SMm (c_{\lambda, \alpha, \beta}(\cdot ,y)-\phi),  \forall y\in Y.\]
We may therefore assume that our minimizing sequence satisfies
\begin{equation}\label{fipsiconjeps1}
\phi^n(x)=\SMn (c_{\lambda^n, \alpha^n, \beta^n} (x, \cdot) -\psi^n), \; \forall x\in X, 
\end{equation}
and 
\begin{equation}\label{fipsiconjeps2}
 \psi^n(y)=\SMm (c_{\lambda^n, \alpha^n, \beta^n}(\cdot ,y)-\phi^n),  \forall y\in Y.
 \end{equation}
From \eqref{fipsiconjeps2}, \eqref{basicsm}, the uniform bounds on $(\lambda^n, \beta^n)$ and \eqref{normalizn} we deduce 
\[\psi^n(y) \leq  \int_X c_{\lambda^n, \alpha^n, \beta^n} (x, y)  \dd \mu(x) -\int_X \phi^n \dd \mu \leq C+ \Vert \alpha^n \Vert_{L^1(\mu)} \Vert g \Vert_{\infty}.\] 
Using the uniform bound \eqref{boundl1alpha} on $\Vert \alpha^n\Vert_{L^1}$, we therefore obtain  a uniform upper bound on $\psi^n$ hence, from \eqref{macrontvc} also a uniform $L^1$ bound on $\psi^n$, i.e. we arrive at  \eqref{boundl1psi}. In a similar way, \eqref{fipsiconjeps1} and the previous bounds imply a pointwise inequality of the form $\phi^n(x) \leq C + \vert \alpha^n(x)\vert \Vert g \Vert_{\infty} + \Vert \psi^n\Vert_{L^1(\nu)}$ which, together with \eqref{boundl1alpha} and \eqref{boundl1psi}  yields the uniform $L^1$ bound \eqref{boundl1phi} on $\phi^n$. 

\smallskip

As before, the crucial point in the proof now is to improve the $L^1$ bound \eqref{boundl1alpha} into an $L^{\infty}$ bound. For this, we use a similar strategy setting
\[\phi_0^n(x):=\SMn (c_{\lambda^n, 0, \beta^n} (x, \cdot) -\psi^n), \; \tal^n(x):=\begin{cases} \alpha^n(x) \mbox{ if  $\phi^n(x) \geq \phi_0^n(x)$}\\ 0 \mbox{ otherwise}\end{cases}  \; \forall x\in X\]
and 
\begin{equation}\label{ctfineps}
\tfi^n(x)=\SMn (c_{\lambda^n, \tal^n, \beta^n} (x, \cdot) -\psi^n)=\max(\phi^n(x), \phi_0^n(x)), \;  \forall x\in X
\end{equation}
since $\tfi^n\geq \phi^n$,  $(\tfi^n, \psi^n, \lambda^n, \tal^n, \beta^n)$ is again a minimizing sequence for \eqref{dualpbwotreleps}. Our uniform bounds on $\lambda^n$, $\beta^n$ and $\psi^n_+$ readily give the uniform lower bound  $\phi_0^n\geq -C$, hence with \eqref{blbounded}, we have
\begin{align*}
-C & \leq \tfi^n(x) =-\eps \log \Big(\int_Y \exp\Big(\frac{-c_{\lambda^n, \tal^n, \beta^n} (x, y) +\psi^n(y)}{\eps}\Big) \dd \nu(y)\Big)\\
&\leq C-\eps\log \Big(\int_Y \exp\Big(\frac{-\tal^n(x) \cdot g(x,y) +\psi^n(y)}{\eps}\Big) \dd \nu(y)\Big).
\end{align*}
Now observe that if $\mu$ is not a Dirac mass (recall that as in the proof of Theorem  \ref{existoptpot}, if $\mu$ is a Dirac mass an $L^\infty$ bound on $\tal^n$ can be directly be deduced from \eqref{normalizn}) for every $n$ and $x\in \spt(\mu)$, Lemma \ref{lemma:moment_on_a_set} guarantees that there exists a measurable set $A\subset Y$ such that for some constant $\rho>0$ (independent of $x$ and $n$), $\nu(A) \geq \rho$ and $\tal^n(x) \cdot g(x,y) \leq -\rho \vert \tal^n(x)\vert$ for every $y\in A$, we thus obtain
\begin{align}
-C & \leq -\eps \log \Big(\int_A \exp\Big(\frac{-\tal^n(x) \cdot g(x,y) +\psi^n(y)}{\eps}\Big) \dd \nu(y)\Big) \nonumber \\
&=-\eps \log(\nu(A)) -\eps \log \Big(\frac{1}{\nu(A)} \int_A \exp\Big(\frac{-\tal^n(x) \cdot g(x,y) +\psi^n(y)}{\eps}\Big) \dd \nu(y)\Big) \nonumber\\
&\leq -\eps \log(\rho)+\frac{1}{\nu(A)}  \int_A  \tal^n(x) \cdot g(x,y) \dd\nu(y) -\frac{1}{\nu(A)} \int_A \psi^n \dd \nu \nonumber\\
&\leq  -\eps \log(\rho) -\rho \vert \tal^n(x)\vert + \frac{1}{\rho} \Vert \psi^n \Vert_{L^1(\nu)} \label{eq:linfty-bound-alpha}
\end{align}
where we used Jensen's inequality in the third line. Together with  \eqref{boundl1psi} we reach the desired uniform $L^\infty$ bound \eqref{talnbound} on $\tal^n$. We have already observed that $\tfi^n \geq \phi_0^n \geq -C$, but from \eqref{talnbound} and $\tfi^n(x) \leq C + \vert \tal^n(x)\vert \Vert g \Vert_{\infty} + \Vert \psi^n\Vert_{L^1(\nu)}$ we deduce a uniform upper bound as well hence deduce that $\Vert \phi^n \Vert_{L^{\infty}(\mu)} \leq C$. We finally replace $\psi^n$ by the minimizer of $\psi \in C(Y) \mapsto (F+G_\eps \circ \Lambda)(\tfi^n, \cdot, \lambda^n, \tal^n, \beta^n)$ which is explicitly given by
\begin{equation}\label{changepsino}
\tps^n(y):= \SMm (c_{\lambda^n, \tal^n, \beta^n}(\cdot ,y)-\tfi^n),  \forall y\in Y.
\end{equation}
By construction $\txi^n:=(\tfi^n, \tps^n, \lambda^n, \tal^n, \beta^n)$ is another minimizing sequence  for \eqref{dualpbwotreleps}. Note that not only $\tps^n$ is uniformly  bounded in $L^\infty$ but it is also uniformly equicontinous thanks to our uniform bounds on  $\lambda^n$, $\tal^n$ , the compactness of $K$ in assumption \textbf{(H3)} and the uniform continuity of $f$ and $g$ on $X\times Y$. More precisely, it easily follows from \eqref{changepsino} that for every $y_1$ and $y_2$ in $Y$, 
\[\begin{split}
 \vert \tps^n (y_1)-\tps^n(y_2) \vert &  \leq  \sup_{\gamma \in K} \vert \gamma(y_1)-\gamma(y_2)\vert \\
 &+ \sup_{x\in X} \{ \Vert \tal^n\Vert_{\infty} \vert g(x,y_1)-g(x,y_2)\vert + \Vert \lambda^n\Vert_{\infty}  \vert f(x,y_1)-f(x,y_2)\vert\}
 \end{split}\]
thanks to the compactness of $K$ and the Arzel\`a-Ascoli Theorem, this shows precompactness of $\tps^n$ in $C(Y)$. We have  reached the very same conclusion as at the end of Step 2 in the proof of Theorem \ref{existoptpot}. The rest of the proof is exactly the same: construct thanks to Mazur's Lemma another minimizing sequence $\hxi^n$ which also remains uniformly bounded and  converges to some $\xi\in E_{\Relax}$ in the very same strong sense as in \eqref{oyep2}-\eqref{oyep3}-\eqref{oyep4}. The fact that $\xi$ solves \eqref{dualpbwotreleps} is obtained in the same way (and applying Lebesgue's dominated convergence Theorem for the convergence of $G_\eps(\Lambda \hxi^n)$ to  $G_\eps(\Lambda \hxi)$).

\end{proof}

\begin{remark} Let us briefly compare Theorem \ref{existoptpoteps} with the results of \cite{NutzWieselEMOT} on the martingale Schr\"{o}dinger bridge problem. In \cite{NutzWieselEMOT}, the authors establish dual attainment for an entropic version of the martingale optimal problem between two measures on the real line under an irreducibility condition and a condition on the supports which is slightly weaker than the nondegeneracy  condition from definition \ref{def:nondegeneracy}. Theorem \ref{existoptpoteps} together with Proposition \ref{prop:stable-slater} and  Theorem \ref{thm:stab-irred-support} grants dual attainment under irreducibility and nondegeneracy but in higher dimensions and also for more general costs than in \cite{NutzWieselEMOT}.
\end{remark}

\begin{remark}
\label{rem:norme_uniforme}
Let us observe that the constant bounding the $L^\infty$ norms of $\phi,\psi,\alpha,\lambda$ and $\beta$ is uniform for $\eps$ in a neighborhood of $0$. Indeed, the key initial estimate to derive the universal upper bound relies on the value of the problem for some competitor given by \textbf{(H2)}. Since the competitor for the entropic case is also valid for the unregularized case this initial estimate behaves linearly in $\eps$ as it vanishes.
Moreover the bound on $\Vert \alpha \Vert_{L^1(\mu)}$ relies on Lemma \ref{lem:slater-Hfinite} which is independent of $\eps$. Finally, the estimate on $\Vert \alpha \Vert_{L^\infty(\mu)}$ also depends linearly on $\eps$ as shown by \eqref{eq:linfty-bound-alpha}.

\end{remark}
\subsection{Primal-Dual optimality conditions}

Now we have dual attainment results at our disposal, we can fully characterize optimal plans.

\begin{theorem}\label{theorem:opticond}
Assume \textbf{(H1)-(H2)-(H3)-(H4)-(H5)}, and denote by $\xi = (\phi,\psi,\al,\be,\la)$ a solution of \eqref{dualpbwotrel}.\\
(i) Then $\al$ is a Lagrange multiplier for the constraint $\int_Y g(x,y) \dd \pi_x(y) = 0$ for $\mu$-a.e. $x$ of the primal problem:
\begin{equation}
    \min_{\pi \in \pigor} J(\pi) = \min_{\pi \in \Pi(\mu,\nu)} \left\{ J(\pi) + \int_{X\times Y} \al(x) \cdot g(x,y) \dd \pi(x,y)\right\}.
\end{equation}
(ii) Let $\pi\in \pigor$, then $\pi$ solves \eqref{pbwot} if and only if:
\begin{itemize}
\item $\pi$ is concentrated on the set $\{(x,y) \in X \times Y : \Lambda \xi (x,y) = 0\}$
\item for $\mu$ a.e. $x$, we have 
\begin{equation}\label{focthetabeta}
\lambda(x)\in \partial \theta (\int_Y f(x,y) \dd \pi_x(y))\; \mbox{and } c_x(\pi_x)+ c_x^*(\beta_x)=\int_Y \beta_x \dd \pi_x.
\end{equation}
\end{itemize}
\end{theorem}

\begin{proof}
(i) Let $\pi \in \Pi(\mu,\nu)$. Proceeding as in \eqref{dualineq} but keeping the term $\int_{X\times Y} \al \odot g \dd \pi$ we have
\begin{equation}\label{diffpdualA}
     F(\xi) = A_1+A_2+ \int_{X\times Y} \La \xi \dd \pi - J(\pi) - \int_{X\times Y} \al \odot g \dd \pi 
\end{equation}
where $A_1$ and $A_2$ are the (non-negative) quantities given by
\begin{align*}
    A_1&=\int_X \Big( c_x(\pi_x)+ c_x^*(\beta_x)-\int_Y \beta_x \dd \pi_x \Big) \dd \mu(x)\\
    A_2&=\int_X \Big( \theta (\int_Y f_x \dd \pi_x)+\theta^*(\lambda(x))-\lambda (x)\cdot  \int_Y f_x \dd \pi_x \Big) \dd \mu(x)
\end{align*}

So we have
\[ -F(\xi) \leq \inf_{\pi \in \Pi(\mu,\nu)} \left\{ J(\pi) + \int_{X\times Y} \al \odot g \dd \pi \right\} \leq \min_{\pi \in \pigor} J(\pi).\] 
where the second inequality is simply a direct consequence of the definition of $\pigor$. But by  \eqref{equalpdvalues} and optimality of $\xi$, the left hand side and the right hand side coincide.

(ii) Let $\pi \in \pigor$, by Corollary \ref{relaxcoro}, we know that since $\xi$ solves \eqref{dualpbwotrel}, $\pi$ solves \eqref{pbwot} if and only if  $F(\xi) + J(\pi) = 0$. Thanks to \eqref{diffpdualA} and the fact that $\int_{X\times Y}  \alpha \odot g \dd \pi=0$ this amounts to the requirement that the three non-negative terms $ \int_{X\times Y} \La \xi \dd \pi$, $A_1$ and $A_2$ vanish i.e. respectively the fact that $\pi$ concentrates on  $\{ \La \xi =0\}$ and \eqref{focthetabeta}. 
\end{proof}

In the entropic case, we can characterize the optimal plan as follows:

\begin{theorem}\label{theorem:opticondeps}
Assume that $Y$ is a compact subset of $\R^d$ and \textbf{(H1)-(H2)-(H3)-(H4)-(H5)}, and denote by $\xi = (\phi,\psi,\al,\be,\la)$ a solution of \eqref{dualpbwotreleps}.\\
(i) Then $\al$ is a Lagrange multiplier for the constraint $\int_Y g(x,y) \dd \pi_x(y) = 0$ for $\mu$-a.e. $x$ of the primal problem:
\begin{equation}
    \min_{\pi \in \pigor} J_\ep(\pi) = \min_{\pi \in \Pi(\mu,\nu)} \left\{J_\ep(\pi) + \int_{X\times Y}  \al(x) \cdot g(x,y) \dd \pi(x,y)\right\}.
\end{equation}
(ii) Let $\pi\in \pigor$, then $\pi$ solves \eqref{pbwoteps} if and only if $\pi$ is absolutely continuous with respect to $\mu \otimes \nu$, with density given by $\exp\left( -\La\xi/\ep\right)$
and \eqref{focthetabeta} holds.
\end{theorem}

\begin{proof}
 Let $\pi \in \Pi(\mu,\nu)$, from \eqref{diffpdualA} and \eqref{ababayrou} we deduce that 
 \[J_\eps(\pi)+ F(\xi) + G_\ep(\La \xi)-\eps +\int_{X\times Y} \alpha \odot g \dd \pi= A_1+A_2+ A_3\]
where $A_1$ and $A_2$ are the same non-negative quantities as in the previous proof and 
\[A_3= G_\eps(\La \xi)+ G_\eps^*(-\pi)+ \int_{X\times Y} \Lambda \xi \dd \pi\]
which is non-negative by Young's inequality. The proof of (i) is then exactly as in the previous theorem. As for (ii) we note that $\pi\in \pigor$ solves \eqref{pbwoteps} if and only if $J_\eps(\pi)+ F(\xi) + G_\ep(\La \xi)-\eps=0$ i.e. when $A_1$, $A_2$ and $A_3$ vanish. It is well-known (see \cite{Varadhan}) that $A_3=0$ means that $\pi$ is absolutely continuous with respect to $\mu \otimes \nu$, with density given by $\exp\left( -\La\xi/\ep\right)$. The condition $A_1=A_2=0$ is equivalent to \eqref{focthetabeta} as before. 
\end{proof}



\section{Convergence of the primal problem}\label{sec-convergence}

We now study the asymptotic behaviour, as the entropic parameter $\ep \downarrow 0$, of the regularized weak optimal transport problem introduced in \eqref{pbwoteps}. Our goals are twofold: (i) to justify that minimizers of the entropically regularized problems converge to minimizers of the original constrained weak
transport problem, and (ii) to quantify the rate at which the regularized optimal values approach the unregularized one.  Besides the purely theoretical interest, these results validate the use of Sinkhorn–type algorithms—whose efficiency is a prime motivation for adding the entropy term—in practical implementations.

In many applications, the hard moment constraint $\int g \, \dd \pi_x = 0$ is itself approximated by a soft penalization.  We therefore treat simultaneously (a) the “pure” entropic regularization and (b) a mixed penalization where the constraint is replaced by a penalty $\ze^{-1}\int \tilde\theta(\int g \,\dd\pi_x)\,\mu(\dd x)$.  A key point will be to
balance the decay of $\ep$ and $\ze$; we identify the mild condition $\ep \ln(\ze)\to 0$ ensuring that both approximations recover the same $\Gamma$-limit and that the optimal values converge with explicit logarithmic rates.

\subsection{Gamma-convergence}

We first record the $\Gamma$-convergence of the functional $J_\ep$
towards $J$. See problems \eqref{pbwoteps} and \eqref{pbwot} for the definitions of $J_\ep$ and $J$.

\begin{proposition}
\label{prop:gamma-simple}
If $Y$ is a compact subset of $\R^d$ and \textbf{(H1)-(H2)-(H3)-(H4)-(H5)} hold, then $J_\ep$, the entropic functional, $\Gamma$-converges 
to $J$ on $\pigor$ endowed with the weak-$\ast$ topology.
\end{proposition}

\begin{proof}
\emph{Preliminary remark} Let $\eta>0$ be as in assumption \textbf{(H5)}. Since assumption \textbf{(H4)}
also holds, by Lemma \ref{lem:slater-Hfinite} there exist a constant $C>0$ ($C = -d\ln(\de)+d\ln\left(\mathrm{diam}_\infty(Y)+2\right) $ as in Lemma \ref{lem:slater-Hfinite}) such that $\| r \|_\infty \leq \frac{\eta}{2}$ implies that $r(x)$ can be represented as $\int g(x,y )\dd \ga_x(y)$ for some $\ga \in \Pi(\mu,\nu)$ with $H(\ga \mid \mu \otimes \nu) \leq C$.

Let $\pi \in \pigor$.
We construct a sequence $(\pi^{(\ep)})_{\ep>0} \in \pigor$ such that $J_\ep(\pi^{(\ep)}) \to J(\pi)$.
Fix $\ep>0$ and let $\pi^\de$ be the $\de$-sliced approximation of $\pi$, with a parameter $\de=\de(\ep)$ to be specified below.

Let $\omega_g$ be the modulus of continuity of $g$ and set $\ta := \frac{2\omega_g(\sqrt{d}\de)}{\eta}$. Consider $\de$ small enough so that $0 < \tau < 1$ and define
\[
r(x) := \frac{\ta-1}{\ta}\int_Y g(x,y)\,\dd \pi_x^\de(y).
\]
Thanks to the properties of the sliced-approximation stated in Lemma \ref{lem:sliced-entropy}, $W_\infty(\pi_x,\pi_x^\delta) \leq \sqrt{d}\delta$ for $\mu$-a.e. $x$, and so we have $\bigl|\int g(x,y)\,\pi^\de(\dd y\mid x)\bigr|\le \omega_g(\sqrt{d}\de)$ for $\mu$-a.e.\ $x$.
Hence by combination with the definition of $\ta$, we have $\|r\|_\infty < \frac{\eta}{2}$, and as seen in the preliminary remark, Lemma~\ref{lem:slater-Hfinite} yields $\gamma\in \Pi(\mu,\nu)$ with $H(\gamma\mid\mu\otimes\nu)<C$ and
$\int g(x,y)\,\dd \gamma_x(y)=r(x)$.
By definition of $r$, this can be expressed as
\[
(1-\ta)\int_Y g(x,y)\,\dd\pi_x^\de(y)
+\ta \int_Y g(x,y)\,\dd \gamma_x(y) = 0.
\]
We set $\pi^{(\ep)} := (1-\ta)\pi^\de + \ta\gamma$.
Then $\int g(x,y)\,\dd \pi_x^{(\ep)}(y)=0$, and by convexity of $\pigor$ we have $\pi^{(\ep)}\in\pigor$.
Moreover, by convexity of the relative entropy and of the squared Wasserstein distance,
\begin{align}
H(\pi^{(\ep)}\mid\mu\otimes\nu) &\le
     (1-\ta)H(\pi^\de\mid\mu\otimes\nu)+\ta H(\gamma\mid\mu\otimes\nu)
     \label{eq:H-conv}\\
W_2^2(\pi^{(\ep)},\pi) &\le
     (1-\ta)W_2^2(\pi^\de,\pi)+\ta W_2^2(\gamma,\pi)
     \label{eq:W2-conv}
\end{align}
We now choose $\de=\ep$. Then $\de$ and $\ta=2\omega(\sqrt{d}\ep)/\eta$ tend to $0$ as $\ep\downarrow0$.
Since $\gamma$ and $\pi$ share the marginals $\mu$ and $\nu$, their $W_2$-distance is bounded in terms of second moments:
$W_2^2(\gamma,\pi) \le 4 m_2(\mu)+4 m_2(\nu)$.
By inequality \eqref{eq:W2-conv} we thus obtain $W_2(\pi^{(\ep)},\pi)\to0$.
Furthermore, $H(\gamma\mid\mu\otimes\nu)\le C$ and $H(\pi^\de\mid\mu\otimes\nu)\le d\ln(1/\de)$, so inequality \eqref{eq:H-conv} implies
$(0 \leq) \, \ep\,H(\pi^{(\ep)}\mid\mu\otimes\nu) \leq d\ep\ln(1/\ep) + C \ep$ so by comparison, $\ep\,H(\pi^{(\ep)}\mid\mu\otimes\nu)$ tends also to 0.
By continuity of $f$ and $\theta$ (assumption \textbf{(H1)}) and compactness of $\supp\mu$,
\[
\int_X \theta\Bigl(\!\int_Y f\,\dd\pi^{(\ep)}_{x}\Bigr)\,\dd\mu(x)
\longrightarrow
\int_X \theta\Bigl(\!\int_Y f\,\dd\pi_x\Bigr)\,\dd\mu(x).
\]
Finally, by combination of assumptions \textbf{(H2)-(H3)}, we have for any $x\in\supp(\mu)$
\[
|c_x(\pi) - c_x(\pi^{(\ep)}) | \leq \omega_K(W_1(\pi,\pi^{(\ep)})) \leq \omega_K(W_2(\pi,\pi^{(\ep)}))
\] 
with $\omega_K$ a concave common modulus of continuity for the functions of the compact $K$ of assumption \textbf{(H3)}. So  by integration and comparison,
\[ \bigl|\int_X c_x(\pi^{(\ep)}_x)\,\dd\mu(x)-\int_X c_x(\pi_x)\,\dd\mu(x)\bigr|\to0. \]
Hence $J_\ep(\pi^{(\ep)})$ converges to $J(\pi)$ as $\ep$ tends to $0$, proving the $\Gamma$-limsup inequality.

The $\Gamma$-liminf inequality follows from the lower semicontinuity of each term and the fact that
$\ep H(\cdot\mid\mu\otimes\nu)\ge0$.
Hence $J_\ep \xrightarrow{\Gamma} J$ on $\pigor$ endowed with the weak-$\ast$ topology.
\end{proof}

In computational practice the moment constraint $\int g \dd \pi_x = 0$ is rarely enforced exactly;
instead one adds a penalty capturing small violations, as for the soft constraint $\int_Y f \dd \pi_x = 0$.
It leads to the following problem:

\begin{equation}
\label{eq:WOT-ep-ze}
    \min_{\pi \in \Pi(\mu,\nu)} J_{\ep,\ze}(\pi) := J_\ep(\pi) 
    + \frac{1}{\ze}\int_X \tth(\int_Y g \dd \pi_x) \dd \mu(x).
\end{equation}
We now show that the penalized problems still approximate the initial problem \eqref{pbwot} provided the penalty ($\ze$) and entropy ($\ep$) parameters vanish in a compatible manner that holds e.g. if $\ze = \ep$.

\begin{proposition}
\label{prop:gamma-mixte}
Consider the functionals $J_{\ep,\ze}$ defined for problem \eqref{eq:WOT-ep-ze} and $J$ is defined in problem \eqref{pbwot}.
Suppose that assumptions \textbf{(H1)-(H2)-(H3)} hold and that $Y$ is a compact subset of $\R^d$. Suppose also that $\tth$ is a convex function defined on $\R^N$ such that $\tth(x)=0$ if and only if $x=0$ and assume that the function $g$ is Hölder continuous. Then, when $\ep,\ze \to 0$ and under the compatibility condition $\ep |\ln(\ze)| \to 0$, the functional $J_{\ep,\ze}$  $\Gamma$-converges to the functional $\pi \mapsto \left\{\begin{array}{cc}
    J(\pi) & \text{ if } \pi \in \pigor   \\
    + \infty & \text{ else. } \end{array}\right.$ on $\Pi(\mu,\nu)$ endowed with the weak-$\ast$ topology.
\end{proposition}

\begin{proof}
The argument mirrors Proposition~\ref{prop:gamma-simple}, except that the $\delta$-sliced approximation itself serves as a competitor. Let
$\pi \in \pigor$. Fix $\ep,\ze>0$ and let $\pi^\delta$
be the $\delta$-sliced approximation of $\pi$, for a parameter $\delta=\delta(\ep,\ze)>0$ to be chosen later.
By property of the sliced-approximation stated in Lemma \ref{lem:sliced-entropy}, $W_\infty(\pi_x,\pi_x^\delta) \leq \sqrt{d}\delta$ for $\mu$-a.e. $x$. Hence, by continuity of f and $\theta$ (assumption \textbf{(H1)}),
\[
\int_X \theta\!\Bigl(\!\int_Y f\,\dd\pi^\delta_x\Bigr)\,\dd\mu(x)\;\longrightarrow\;
\int_X \theta\!\Bigl(\!\int_Y f\,\dd\pi_x\Bigr)\,\dd\mu(x)\qquad\text{as }\delta\to0.
\]
With the same method as in the previous proof, assumptions \textbf{(H2)-(H3)} grant
\[
\int_X c_x(\pi^\delta_x)\,\dd\mu(x)\;\longrightarrow\;\int_X c_x(\pi_x)\,\dd\mu(x)\qquad\text{as }\delta\to0.
\]

For the entropy term, by property of the sliced-approximation stated in Lemma \ref{lem:sliced-entropy}, we have $\ep H(\pi^\delta\mid\mu\otimes\nu)\le d\,\ep|\ln\delta|$, so
$\ep H(\pi^\delta\mid\mu\otimes\nu)\to0$ provided $\ep|\ln\delta|\to0$.

Concerning the penalization term, since $g$ is Hölder continuous and $\tilde\theta$ is (locally) Lipschitz
(as argued earlier for $\theta$), there exist positive constants $C_1,a>0$ such that, whenever $\de$ is small enough, for $\mu$-a.e. $x$,
\[
 \tilde\theta\!\Bigl(\!\int_Y g\,\dd\pi^\delta_x\Bigr)\,
\;\le\; C_1\,W_\infty(\pi_x^\delta,\pi_x)^{a}\;\le\; C_1 (\sqrt{d}\,\delta)^{a},
\]
where the last inequality holds from Lemma \ref{lem:sliced-entropy}. Thus $
\frac{1}{\ze}\int_X \tilde\theta\!\left(\!\int_Y g\,\dd\pi^\delta_x\right)\,\dd\mu(x)\to0$ provided $\delta^{a}/\ze\to0$. Choose $\delta := \ze^{2/a}$. Then $\delta^{a}/\ze = \ze \to 0$, and
$\ep|\ln\delta| = \frac{2}{a}\,\ep\,|\ln\ze|\,\to\,0$
by compatibility condition. Therefore both the entropy term and the penalization term converge to 0, so $\pi^\delta$ is a recovery sequence, proving the $\Gamma$-limsup inequality.

The $\Gamma$-liminf inequality follows from lower semicontinuity of each term and the nonnegativity of
$\ep H(\cdot\mid\mu\otimes\nu)$ and $\ze^{-1}\tilde\theta(\cdot)$. So $\Gamma$-convergence holds.
\end{proof}

\subsection{Rates of convergence}
We complement the qualitative $\Gamma$-convergence with quantitative bounds on the
optimal value gaps. 
With the notation $v_{\ep,\ze}$ for the value of the problem \eqref{eq:WOT-ep-ze}, that is $v_{\ep,\ze} = \min_{\pi \in \Pi(\mu,\nu)} J_{\ep,\ze}(\pi)$, we have by monotonicity in $\ep$ and $\ze$
\begin{equation}
   v_{0,\ze} \leq v_{\ep,\ze}, v_{0,0} \leq v_{\ep,0}.
\end{equation}
By the triangle inequality,
\begin{equation}
    |v_{\ep,\ze} - v_{0,0}| \leq  |v_{0,\ze} - v_{0,0}| + |v_{\ep,0} - v_{0,0}|.
\end{equation}
Thus it suffices to estimate the two error terms $v_{0,\ze}-v_{0,0}$ and $v_{\ep,0}-v_{0,0}$
separately. This is the content of the next two propositions. In the purely entropic case,
as in classical (strong) optimal transport (see for example \cite{carlier2022convergencerategeneralentropic}), the leading order of the error is $\ep \ln(1/\ep)$.

\begin{proposition}\label{prop:rate-ep}
    Suppose that the functions $c_x$ are equi-Hölder-continuous for $W_\infty$. Further assume that $Y$ is a compact subset of $\R^d$, that assumptions \textbf{(H1)-(H2)-(H3)-(H4)-(H5)} hold and that the functions $f$, $g$ are  Hölder continuous. Then there exists a constant $C>0$ such that for $\ep>0$ small enough, we have
    \[ 0 \leq \min_{\pi \in \pigor} J_\ep(\pi)\,\, - \min_{\pi \in \pigor} J(\pi) \,\, \leq C\ep\ln\left(\frac{1}{\ep}\right).\]
\end{proposition}

\begin{proof}
By duality (see Theorem \ref{theorem:opticond}) we can write
\begin{equation}\label{eq:semidual_convergence}
\min_{\pi \in \pigor} J_\ep(\pi)
    = \inf_{\pi \in \Pi(\mu,\nu)}
      J_\ep(\pi) 
        + \int_{X\times Y} \alpha^\ep(x)\, g(x,y)\,\dd\pi(x,y),
\end{equation}
where $\alpha^\ep$ is an optimal dual variable corresponding to the constraint $\int g(x,y) \dd \pi_x(y) = 0 \, \mu(x)-a.e.$. The existence of such an optimal $\alpha^\ep \in L^\infty(X,\mu)$ is guaranteed by Theorem \ref{existoptpoteps}. Moreover, following remark \ref{rem:norme_uniforme}, the functions $\al_\ep$ are uniformly (in $\ep$) bounded in $L^\infty(\mu)$.

Let $\ga$ be an optimal plan for the unregularized problem \eqref{pbwot}. For $\delta>0$, let $\ga^\delta$ be the $\delta$-sliced approximation of $\ga$.
Using $\ga^\delta$ as a competitor in the RHS of equation \eqref{eq:semidual_convergence}, we get
\begin{equation}
\label{eq:competitor}
\min_{\pi \in \pigor} J_\ep(\pi) \leq J_\ep(\ga^\de) + \int_{X\times Y} \alpha^\ep(x)\, g(x,y)\,\dd\ga^\delta(x,y).
\end{equation}
By Hölder-continuity of $f$ and since $\th$ is Lipschitz as seen before by assumption \textbf{(H1)}, the function $\nu \mapsto \th(\int f \dd \nu)$ is Hölder-continuous for $W_\infty$, and by sum with $c_x$ that also has this property, we obtain
\[ \left\lvert (c_x(\ga^\de_x) + \th (\int_Y f \dd \ga^\de_x)) - (c_x(\ga_x) + \th (\int_Y f \dd \ga_x) \right\rvert \leq C_1 W_\infty(\ga^\de_x,\ga_x)^a,
\]
with $C_1,a>0$ independent of $x$ by hypothesis. So integrating this inequality, we get
\[ J(\ga^\de) - J(\ga) \leq C_1 \int_X W_\infty(\ga^\de_x,\ga_x)^a \dd \mu(x).
\]
From the properties of the sliced-approximation described in Lemma \ref{lem:sliced-entropy}, $W_\infty(\ga^\de_x,\ga_x) \leq \sqrt{d}\de$ for $\mu$-a.e. $x$. So $J(\ga^\de) - J(\ga) \leq C_2 \de^a$, with $C_2 = C_1 \sqrt{d}^a$. Moreover $J_\ep(\ga^\de) = J(\ga^\de) + \ep H(\ga^\de \mid \mu \otimes \nu)$ and we have, still from Lemma \ref{lem:sliced-entropy}, $H(\gamma^\delta \vert \mu \otimes \nu)\leq -d \log(\delta) + C_3$ with $C_3 := d \log({\mathrm{diam}}_\infty Y+2)$. So we obtain
\[
J_\ep(\ga^\de) \leq  J(\ga) + C_2 \de^a - d\ep \ln(\de) + C_3 \ep.
\]
This is a sharp bound for the first term in the RHS of inequation \eqref{eq:competitor}. As for the second term, since $(\al_\ep)_\ep$ is uniformly (in $\ep$) bounded in $L^\infty(\mu)$ and using the property of the sliced-approximation $W^\infty(\pi_x^\de,\pi_x) \leq \sqrt{d}\de$ as well as the Hölder-continuity of the function g, we obtain
\[ \left\lvert \int_{X\times Y} \alpha^\ep(x)\, g(x,y)\,\dd\ga^\delta(x,y) \right\rvert \leq C_4\de^{b}
\]
with $C_4,b >0$. With those two bounds for the RHS of \eqref{eq:competitor}, and replacing $J(\ga)$ by 
$\displaystyle \min_{\pi \in \pigor} J(\pi)$, since it was chosen as a minimizer, we get
\[
    \min_{\pi \in \pigor} J_\ep(\pi) - \min_{\pi \in \pigor} J(\pi) \leq C_2 \delta^{a} - d\,\ep\ln \delta +C_3 \ep + C_4 \de^{b}.
\]
The claim follows, by adjusting $\de$ to $\ep^{\frac{1}{a\wedge b}}$ and with $C := C_2 + C_3 + C_4 + \dfrac{d}{a \wedge b}$.
\end{proof}

Having controlled the purely entropic error in Proposition~\ref{prop:rate-ep}, we now turn to the penalty-only regime ($\ep=0$, $\ze>0$).
In this case the entropy disappears and the rate is dictated by the coercivity of $\tth$ around $0$. The next result shows that the relaxation
error decays algebraically with $\ze$.

\begin{proposition}\label{prop:rate-zeta}
    Suppose that $\tth$ is a convex function with $\tth(x) \geq c |x|^{1+a}$ for some constants $c,a>0$.
    Then, under assumptions \textbf{(H1)-(H2)-(H3)-(H4)-(H5)}, we have
    \[ 0 \leq  \min_{\pi \in \pigor} J(\pi) - \min_{\pi \in \Pi(\mu,\nu)} J_{0,\ze}(\pi)  \leq C\ze^{1/a} \]
    for some constant $C>0$.
\end{proposition}

\begin{proof}
    We have the dual representation
    \[
    \min_{\pi \in \pigor} J(\pi)
    = \inf_{\pi \in \Pi(\mu,\nu)} \left\{
       J(\pi)
        + \int_{X\times Y} \alpha(x)\,g(x,y)\,\dd\pi(x,y) \right\},
\]
where $\alpha$ is an optimal dual variable corresponding to the constraint $\int g(x,y) \dd \pi_x(y) = 0 \, \mu(x)-a.e.$. The existence of such an optimal $\alpha \in L^\infty(X,\mu)$ is guaranteed by Theorem \ref{existoptpot}. Let $\ga$ be an optimal plan for the problem $\inf_{\pi \in \Pi(\mu,\nu)} J_{0,\ze}(\pi)$. Using $\ga$ as a competitor in the above dual expression above yields
\begin{align*}
\min_{\pi \in \pigor} J(\pi)
&\le J(\ga)
   + \int_{X\times Y} \alpha(x) g(x,y)\,\dd\ga(x,y) \\
&=   J_{0,\ze}(\ga)
   + \int_{X\times Y} \alpha(x) g(x,y)\,\dd\ga(x,y)
   - \frac{1}{\ze}\int_X \tilde\theta\!\Bigl(\!\int_Y g\,\dd\ga_x\Bigr)\,\dd \mu(x).
\end{align*}
Apply Fenchel–Young’s inequality $\tilde\theta(u)+\tilde\theta^*(v)\ge u\cdot v$ with
$u=\int_Y g\,\dd\ga_x$ and $v=\ze\,\alpha(x)$:
\[
\int_{X\times Y} \alpha(x) g(x,y)\,\dd\ga(x,y)
\le \frac{1}{\ze}\int_X \tilde\theta\!\Bigl(\!\int_Y g\,\dd\ga_x\Bigr)\,\dd \mu(x)
   + \frac{1}{\ze}\int_X \tilde\theta^*\!\bigl(\ze\,\alpha(x)\bigr)\,\dd \mu(x).
\]
Substituting this into the previous inequality gives
\[
    \min_{\pi \in \pigor} J(\pi) - J_{0,\ze}(\ga)
    \le \frac{1}{\ze}\int_X \tilde\theta^*\!\bigl(\ze\,\alpha(x)\bigr)\,\dd \mu(x).
\]

From $\tilde\theta(x)\ge c|x|^{1+a}$ it follows that
$\tilde\theta^*(v) \le c^{-1}|v|^{\frac{1+a}{a}}$.
Hence, by optimality of $\ga$ for the problem $\inf_{\pi \in \Pi(\mu,\nu)} J_{0,\ze}(\pi)$, we have
\[
    \min_{\pi \in \pigor} J(\pi) - \min_{\pi \in \Pi(\mu,\nu)} J_{\ep,\ze}(\pi) 
    \le C\,\ze^{\frac{1}{a}},
\]
with $C=c^{-1} \|\alpha\|_{\infty}^{\frac{1+a}{a}}$. This proves the claim.
\end{proof}

\begin{remark}

As noted above, the rate of convergence for the mixed problem \eqref{eq:WOT-ep-ze}
is bounded by the sum of the rates given by propositions \ref{prop:rate-ep} and \ref{prop:rate-zeta}.
Under the assumptions of Propositions~\ref{prop:rate-ep} and~\ref{prop:rate-zeta}, we obtain
\[
    \bigl|\min_{\pi \in \Pi(\mu,\nu)} J_{\ep,\ze}(\pi) - \min_{\pi \in \pigor} J(\pi) \bigr|
    \le C_1\,\ep \ln\!\Bigl(\tfrac{1}{\ep}\Bigr) + C_2\,\ze^{1/a},
\]
with some constants $C_1,C_2,a>0$, the parameter $a>0$ depending specifically on the function $\tth$.
\end{remark}

\section{Numerical simulations with SISTA}\label{sec-sista}

In this section, we present the numerical results obtained by implementing a Sinkhorn type algorithm. The algorithm we use has been introduced in \cite{CDGSSista}
and is known as SISTA (S for Sinkhorn and ISTA for Iterative Shrinkage-Thresholding Algorithm, aka proximal gradient descent). See \cite{gallouët2025metricextrapolationwassersteinspace} for insights on the convergence of the algorithm. To streamline the presentation of the algorithm we make the following simplifying assumptions: 
\begin{itemize}
    \item there is $T \in \mathbb{N}^\ast$ and $(x_i)_{i=1}^T,(y_i)_{i=1}^T \in \mathbb{R}^{dT}$ 
    \begin{equation*}
        \mu = \frac{1}{T} \sum_{i=1}^T \delta_{x_i}, \quad \nu = \frac{1}{T} \sum_{i=1}^T \delta_{y_i},
    \end{equation*}
    \item there is $\cost \in C(\mathbb{R}^d \times \mathbb{R}^d)$ such that $c_x(p) = \int \cost(x, y) \dd p(y)$,
    \item $\theta^\ast$ is differentiable.
\end{itemize}
 These two assumptions are such that the optimization problem is now finite dimensional and the dual variable $\beta$ is fixed to be $\cost$. The idea is to solve the dual problem \eqref{dualpbwotreleps}
\[
\sup_{\xi=(\phi,\psi,\lambda,\alpha)} 
\int \phi\, \dd \mu + \int \psi \,\dd \nu 
- \int_X\theta^\ast(\lambda(x))\,\dd\mu(x)
-\varepsilon \int_{X\times Y} e^{-\Lambda \xi(x,y)/\varepsilon}\,\dd(\mu\otimes\nu),
\]
with $\Lambda \xi = \cost + \lambda\odot f + \alpha\odot g - (\phi\oplus\psi)$. The method is based on iterative updates of the dual variables $\phi, \psi, \al$ and $\la$. As for $\ph$ and $\ps$, corresponding to the marginal constraints, we use Sinkhorn iterates, that is exact maximization in $\phi$ and then $\ps$. On the other hand, the variables $\al$ and $\la$ will be updated using a gradient descent. It leads to the following algorithm:

\begin{algorithm}[H]
\caption{SISTA (Sinkhorn–ISTA) for weak OT with moment constraints}
\label{algo:sista}
\begin{algorithmic}[1]
\State \textbf{Input:} $\mu,\nu$, step size $\tau>0$, entropic scale $\varepsilon>0$, initial $\xi = (\alpha,\lambda,\phi,\psi)$
\State \textbf{Repeat until convergence:}
\State \hspace{0.5em}\textbf{(S)} \emph{Sinkhorn block:} 
\begin{align*}
\phi(x) &\leftarrow  \phi(x) + \varepsilon\log\!\int_Y \exp\!\Big(\tfrac{\Lambda\xi(x,y)}{\varepsilon}\Big)\,\dd\nu(y),\\
\psi(y) &\leftarrow \psi(y) + \varepsilon\log\!\int_X \exp\!\Big(\tfrac{\Lambda\xi(x,y)}{\varepsilon}\Big)\,\dd\mu(x),
\end{align*}
\State \hspace{0.5em}\textbf{(I)} \emph{ISTA block:}
\begin{align*}
    \la(x) &\leftarrow  \la(x) - \tau \left[\nabla \theta^\ast(\la(x)) - \int_Y f(x,y) e^{-\Lambda \xi(x,y)/\varepsilon}\,\dd\nu(y)\right]\\
\al(x) &\leftarrow \al(x) + \tau  \int_Y g(x,y) e^{-\Lambda \xi(x,y)/\varepsilon}\,\dd\nu(y).
\end{align*}
\end{algorithmic}
\end{algorithm}

Note that the updates in $\phi$ and $\psi$ correspond exactly to
Sinkhorn's algorithm with the shadow transport cost $c_{\lambda, \alpha, \beta}$ defined in equation \eqref{ozzycost}. In case $\theta^\ast$ is not differentiable the gradient step on $\theta$ is replaced by a proximal step which gives its name to the ISTA step initially presented in \cite{CDGSSista}. Theorem \ref{existoptpoteps} shows that the optimizer can be searched for in a ball of known a priori radius. This ensures that the functional enjoys a Lipschitz gradient on this ball. Thus for a sufficiently small step size $\tau$ the algorithm will converge. We refer to \cite{gallouët2025metricextrapolationwassersteinspace} for a careful analysis of the theoretical convergence of the algorithm.

In the context of martingale optimal transport this algorithm differs from the on presented in \cite{chen2024convergencesinkhornsalgorithmentropic} where the ISTA block \textbf{(I)} is replaced by an exact minimization of the functional in $\al$ with the other variables fixed.

\paragraph{Interpolation between Brenier and Strassen.} Following the idea presented in \cite{Gozlan-Juillet}, we propose to compute for $d=1$ and $t\in[0,1]$
\begin{equation*}
    \inf_{\pi \in \Pi(\mu,\nu)} t \int \vert x-y \vert^2 \dd \pi(x,y) + (1-t) \int \vert x- \int y \dd \pi_x(y) \vert^2 \dd \mu(x).
\end{equation*}
When $t=0$ any martingale coupling will be a solution, and if there is no martingale coupling the solution will be such that the law of its conditional expectation is the projection of the first marginal onto the set $\{\cdot \leq_{\mathrm{cvx}} \nu \}$. When $t=1$ the classical optimal transport problem is recovered. In order to compute the optimal plan,  we will add entropy with a small regularization parameter $\eps > 0$.

Figure \ref{fig:brestra} represents the coupling $\pi$ for different values of $t$. The marginal $\mu$ is the law of the standard Gaussian $\mathcal{N}(0,1)$ on the real line and $\nu$ is the mean of the laws of $\mathcal{N}(-1,1)$ and $\mathcal{N}(1,1)$. The measures are approximated using empirical measures with $T= 1000$ points. The regularization parameter $\eps$ is taken to be $10^{-2}$.

\begin{figure}[htbp]
  \centering
  \includegraphics[width=\textwidth,trim={0cm 4.6cm 0.1cm 4.7cm},clip]{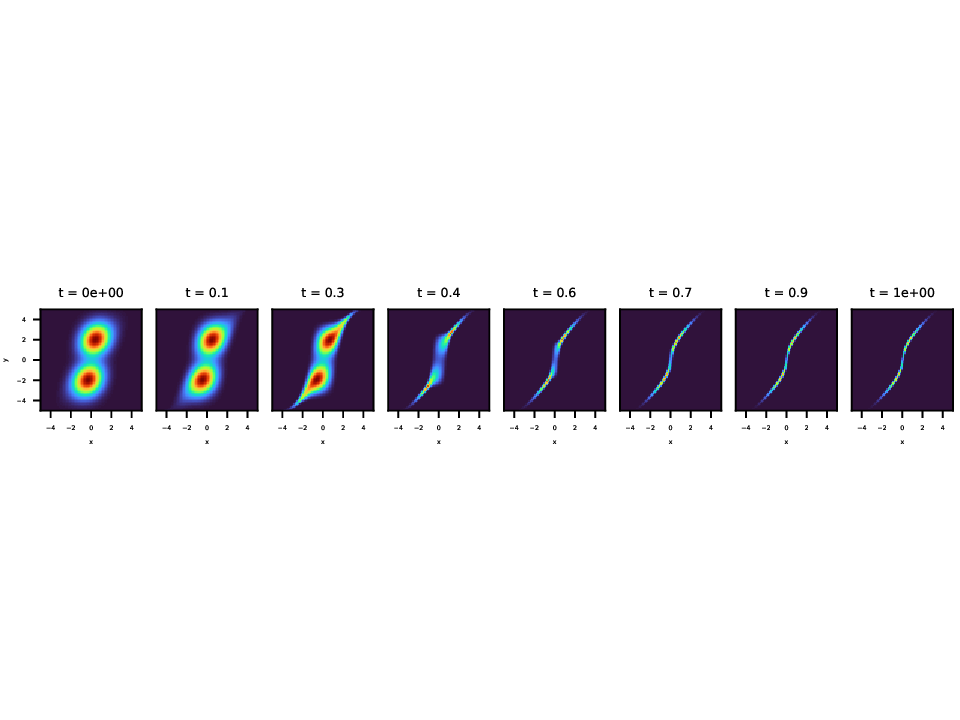}
  \caption{Interpolation between Brenier and Strassen.}
  \label{fig:brestra}
\end{figure}

Observe that for $t= 0$ the optimal coupling is not unique for $\eps = 0$ as any martingale coupling is solution. However, the entropy term added for the algorithm selects the one the smallest relative entropy. The diffusion behaviour which vanishes as $t \to 1$ is not related to $\eps$ rather it is tied to the fact that the barycentric map enjoys a Lipschitz behaviour which degenerates as $t \to 1$.

\paragraph{Interpolation between left-curtain and optimal transport.} For $d=1$, a distinguished coupling in martingale optimal transport is the left-curtain coupling introduced in \cite{BeigJuilletl}. Here we propose to compute an interpolation between this coupling and the optimal transport plan. To do so introduce for $t\in [0,1]$

\begin{equation*}
    \inf_{\pi \in \Pi(\mu,\nu)}  \int t \vert x-y \vert^2  + (1-t) (1+\tanh(-x))\sqrt{1+y^2} \dd \pi(x,y) + \left(\frac{1}{t}-1\right)\int \vert x-\int y \dd\pi_x(y) \vert^2\dd \mu(x).
\end{equation*}

The choice of cost for $t = 0$ is guided by \cite[Proposition 2.8]{MR3573297} which ensures that the left-curtain is optimal under martingale constraint. Proposition \ref{prop:gamma-mixte} ensures that for small values of $t$ this problem $\Gamma$-converges towards the left-curtain problem.

The numerical setting is the same as above: $\mu \sim \mathcal{N}(0,1)$ and $\nu \sim \frac{1}{2}\left(\mathcal{N}(-1,1) + \mathcal{N}(1,1)\right)$ are approximated by empirical measures with $T= 1000$ points and $\eps = 10^{-2}$.

\begin{figure}[htbp]
  \centering
  \includegraphics[width=\textwidth,trim={0cm 4.6cm 0.1cm 4.7cm},clip]{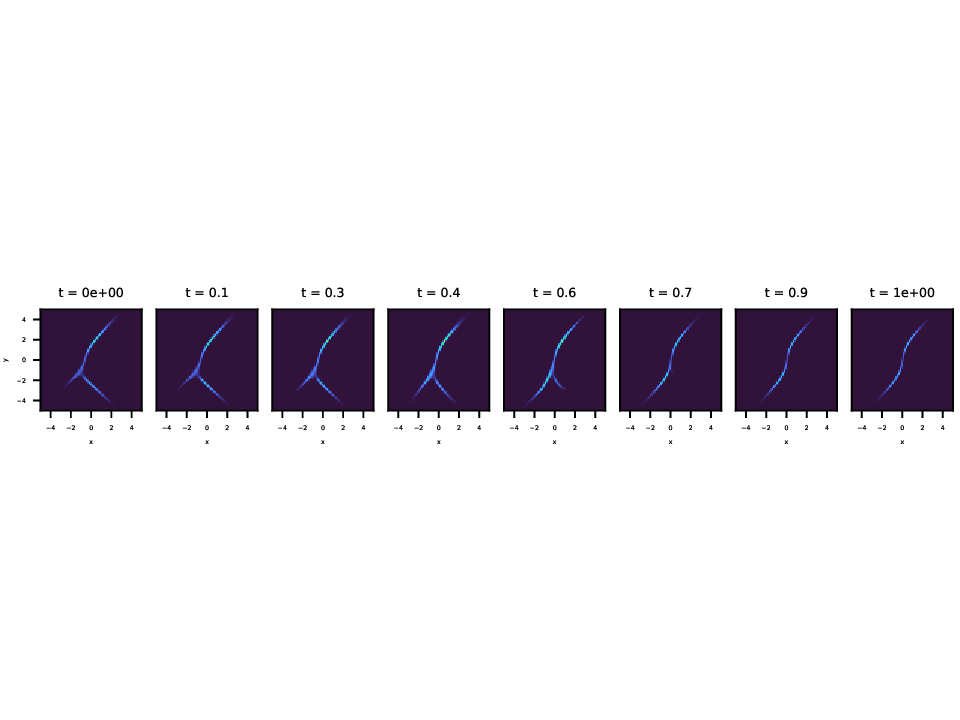}
  \caption{Interpolation between the Left curtain transport and optimal transport}
  \label{fig:lc}
\end{figure}

Observe that the structure of the left-curtain coupling is recovered for small values of $t$.

\section*{Appendix}

\section*{ A. Median of a function in $L^1(\mu)$}

\begin{lemma}\label{lem:median}
For every $\alpha \in L^1((X,\mu), \R^N)$, there exist $\overline{a}\in \R^N$ and $\overline{r} \in L^\infty((X,\mu), \R^N)$ such that $\Vert \overline{r} \Vert_\infty \leq 1$, $\int_X \overline{r} \dd \mu = 0$ and 
\begin{align}
\inf_{a \in \mathbb{R}^N} \Vert \alpha -a \Vert_{L^1(\mu)}&=    \Vert \alpha - \overline{a} \Vert_{L^1(\mu)}  = \int_X \langle \alpha, \overline r \rangle \dd\mu \label{mediandual}\\&=\max  \left\{ \int_X \langle \alpha,r \rangle \dd\mu \; : \;  \Vert r \Vert_{\infty} \leq 1, \; \int_X r \dd \mu=0\right\}.
\nonumber
\end{align}
Any  $\overline{a}\in \R^N$ fulfilling the above condition is called   a \emph{median} of $\alpha$.
\end{lemma}

\begin{proof} 
Note that by the triangle inequality $\Vert \alpha -a \Vert_{L^1(\mu)}\geq \vert a\vert- \Vert \alpha \Vert_{L^1(\mu)}$, the existence of a minimizer (a median) for the minimization problem in the left hand side of \eqref{mediandual} easily follows. Now denoting by $\Lambda \in \L_c(\R^N, L^1((X, \mu), \R^N)$ the canonical imbedding of constants into $L^1$ functions, its adjoint $\Lambda^*\in \L_c( L^{\infty}((X, \mu), \R^N)), \R^N)$ is by construction given by $\Lambda^* r=\int_X r \dd \mu$. Defining $G_\alpha(l):=\Vert \alpha-l \Vert_{L^1(\mu)}$ for every $l\in L^1((X, \mu), \R^N)$,   we readily see  that for every $s\in L^{\infty}((X, \mu), \R^M)$
 \[G_{\alpha}^*(s)= \begin{cases} \int_X  \langle \alpha, s \rangle \dd \mu \mbox{ if $ \Vert s \Vert_{\infty} \leq 1$}\\ + \infty \mbox{ otherwise }\end{cases} \]
hence a direct application of the Fenchel-Rockafellar Theorem gives
\begin{align*}
 \inf_{a \in \mathbb{R}^N} \Vert \alpha -a \Vert_{L^1(\mu)}&=\inf_{a \in \R^N} G_\alpha(\Lambda a)
 = \max\{ -G_{\alpha}^*(-r) \; : \; r\in L^{\infty}((X, \mu), \R^M), \; \Lambda^* r =0\}\\
& =\max  \left\{ \int_X \langle \alpha,r \rangle \dd\mu \; : \;  \Vert r \Vert_{\infty} \leq 1, \; \int_X r \dd \mu=0\right\}.
    \end{align*}

\end{proof}

\section*{B. Proof of Lemma \ref{dualthetarep}}

Let $\pi=\mu \otimes \pi_x \in \Pi(\mu, \nu)$, define  $h(x):=\int_{X} f(x,y) d\pi_x(y)$ for every $x\in X$ then $h$ is Borel and $\Vert h \Vert_{\infty} \leq \Vert f \Vert_{\infty}$. Since $\int_{X\times Y} (\lambda\odot f ) \dd\pi = \int_X \lambda \cdot h  \; \dd \mu$ we have to show that $a=b=c$ where
\[a:=\int_X \theta(h(x)) \dd \mu(x),\;  b:=\sup_{\lambda \in C(X, \R^M)} \left\{ \int_X \lambda \cdot h  \; \dd \mu- \int_X \theta^*(\lambda(x)) \dd \mu(x)\right\}\]
and 
\[c:=\sup_{\lambda \in C(X, \R^M), \Vert  \lambda \Vert_{\infty} \leq C} \left\{ \int_X \lambda \cdot h \;  \dd \mu- \int_X \theta^*(\lambda(x)) \dd \mu(x)\right\}\]
and $C=C(\theta, f)$ is the Lipschitz constant of $\theta$ on $B(0, \Vert f \Vert_{\infty}+1)$. Obviously $b \geq c$ and by integrating Young's inequality $\theta(h(x)) \geq \lambda(x) \cdot h(x)-\theta^*(\lambda(x))$ we readily see that $a\geq b$. It remains to show that $a\leq c$, for this we proceed by regularization as follows. 
Fix  first $\eps>0$, By Lusin's theorem there exists $h_\eps \in C(X, \R^M)$ and $X_\eps\subset X$, compact such that
\begin{equation}\label{mercilusin1}
\Vert h_\eps \Vert_{\infty} \leq \Vert h \Vert_\infty \leq \Vert f \Vert_\infty, \; h=h_{\eps} \mbox{ on $X_\eps$ and } \mu(X \setminus X_\eps) \leq \eps.
\end{equation}
In particular we have
\[ a \leq \int_X \theta(h_\eps(x)) \dd \mu(x)+ L\eps \mbox{ with }  L:= 2 \sup \{\vert \theta(u)\vert, \; u\in \R^M, \vert u \vert \leq \Vert f \Vert_{\infty}\}.\]
Now let $\delta>0$, define the Moreau-Yosida regularization of $\theta$ by
\[\theta_\delta(u):=\min_{v\in \R^M} \{\frac{1}{2 \delta} \vert u-v\vert^2 + \theta(v)\}\]
then $\theta_\delta$ is $C^{1,1}$  (we refer to the seminal paper of \cite{Moreau}) for the properties of Moreau-Yosida approximations used in this proof)
and  $\theta_\delta$ converges monotonically pointwise to $\theta$ as $\delta \to 0^+$ hence also uniformly on compact sets by \textbf{(H1)} and Dini's Theorem, which implies in particular 
\begin{equation}\label{mercimoreau}
a \leq \int_X \theta_{\delta}(h_\eps(x)) \dd \mu(x) + L\eps + C_{\delta} \mbox{ with } \lim_{\delta \to 0^+} C_\delta=0.
\end{equation}
Now define $\lambda_{\eps, \delta}(x):=\nabla \theta_{\delta}(h_\eps(x))$ observe that $\lambda_{\eps, \delta}$ is continuous and for every $x\in X$ we have
\begin{align*}
 \theta_{\delta}(h_\eps(x))= \lambda_{\eps, \delta}(x) \cdot h_\eps(x)-\theta_{\delta}^*(\lambda_{\eps, \delta}(x))\\
 \leq  \lambda_{\eps, \delta}(x) \cdot h_\eps(x)-\theta^*(\lambda_{\eps, \delta}(x))
 \end{align*}
 where we used the fact that $\theta_{\delta}^* \geq \theta^*$ (as a consequence of $\theta_{\delta} \leq \theta$) in the second line. Note also (again we refer to \cite{Moreau}) that
 \[\lambda_{\eps, \delta}(x):=\nabla \theta_{\delta}(h_\eps(x)) \in \partial \theta (v_{\delta}(h_\eps(x)))\]
 where 
 \[v_\delta (h_\eps(x))=\argmin_{v\in \R^M}  \{\frac{1}{2 \delta} \vert h_\eps(x)-v\vert^2 + \theta(v)\}\]
 so that (using $\theta \geq 0$)
 \[  \vert h_\eps(x)- v_\delta (h_\eps(x))\vert^2 \leq 2 \delta \theta(h_\eps(x)) \leq L \delta \]
  in particular for $\delta$ small enough, $v_\delta(h_\eps(x))$ belongs to $B(0, \Vert f \Vert_{\infty)}+1)$ for every $\eps$ and every $x$ so that $\lambda_{\eps, \delta}(x) \in \partial \theta(v_{\delta}(h_\eps(x)))$ has norm less than $C=C(\theta, f)$. For $\delta >0$ small enough, we thus have
 \begin{align*}
 a &\leq \int_{X} (\lambda_{\eps, \delta(x)}  \cdot h_\eps(x)-\theta^*(\lambda_{\eps, \delta}(x)) )\dd \mu(x) + L\eps + C_\delta\\
&= \int_{X} (\lambda_{\eps, \delta}(x) \cdot h(x)-\theta^*(\lambda_{\eps, \delta}(x)) )\dd \mu(x) +  \int_X \lambda_{\eps, \delta}(x) \cdot (h_\eps(x)-h(x)) \dd \mu(x)  +L\eps + C_\delta\\
  &\leq c +\int_X \lambda_{\eps, \delta} (x) \cdot (h_\eps(x)-h(x))\dd \mu(x)  +L\eps + C_\delta
  \end{align*}
recalling \eqref{mercilusin1}  and using the fact that   $\Vert  \lambda_{\eps, \delta}\Vert_{\infty} \leq C$, letting $\eps, \delta \to 0^+$, we obtain the desired inequality $a\leq c$.

\section*{C. Proof of Lemma \ref{dualjeanpierre}} 

The fact that 
\[\int_{X} c(x, \pi_x) \dd \mu(x)\geq  \sup_{\beta \in C(X,K)} \left\{\int_{X\times Y} \beta \dd \pi-\int_X c_x^*(\beta_x) \dd \mu(x) \right\}\]
follows directly from Young's inequality $c_x(\pi_x)\geq \int_Y \beta_x \dd \pi_x-c_x^*(\beta_x)$ whenevener $\beta_x \in C(Y)$. To show the converse inequality, we shall proceed by approximation as before. By compactness of $K$, for every $\delta>0$, we can cover $K$ by finitely many closed balls of radius $\delta$ (for the uniform norm), denote by $(\gamma_{1}, \ldots, \gamma_{N_\delta}) \in K^{N_{\delta}}$ the centers of these balls and by $K_\delta$ the convex hull of the finite set $\{\gamma_1, \ldots, \gamma_{N_\delta}\}$. Thanks to \eqref{cstarlip1}, and assumption \textbf{(H3)} we see that for every $x$
\begin{align*}
c_x(\pi_x)&=\max_{\gamma \in K} \left\{\int_Y \gamma \dd \pi_x-c_x^*(\gamma)\right\} \leq \max_{\gamma \in K_\delta} \left\{\int_Y \gamma \dd \pi_x-c_x^*(\gamma)\right\}+ 2 \delta\\
&=  \max_{\eta \in \Delta_\delta} \left\{\sum_{i=1}^{N_\delta}  \eta_i \int_Y \gamma_i  \dd \pi_x-c_x^*\Big( \sum_{i=1}^{N_\delta}  \eta_i\gamma_i\Big)\right\}   + 2\delta
\end{align*}
where $\Delta_\delta$ is the simplex $\{(\eta_1, \ldots, \eta_{N_\delta}) \in \R_+^{N_\delta} \; : \; \sum_{i=1}^{N_\delta} \eta_i=1\}$. Thanks to Lemma \ref{cstarisnormal}, we can now invoke a measurable selection argument (like Theorem 1.2 in Chapter VIII of \cite{EkelandTemam}) and select a maximizer of $\eta \in \Delta_{\delta} \mapsto \left\{\sum_{i=1}^{N_\delta}  \eta_i \int_Y \gamma_i  \dd \pi_x-c_x^*\Big( \sum_{i=1}^{N_\delta}  \eta_i\gamma_i\Big)\right\}$ which depends in a measurable way on $x$, denote by $\eta^{\delta}(x)=(\eta_1^\delta(x), \ldots, \eta_{N_\delta}^\delta(x))$ such a maximizer and set
\[\beta^{\delta}(x,y):=\sum_{i=1}^{N_{\delta}} \eta_i^{\delta}(x) \gamma_i(y), \; \forall (x,y)\in X\times Y.\]
Now given $\eps>0$, by Lusin's theorem we can find a compact subset $X_{\delta, \eps}$ of $X$ and  $\eta^{\delta, \eps}=(\eta^{\delta, \eps}_1, \ldots \eta^{\delta, \eps}_{N_{\delta}})\in C(X, \R^{N_{\delta}})$ such that 
\begin{equation}\label{mercilusin2}
\eta^{\delta, \eps} =\eta^{\delta} \mbox{ on $X_{\delta, \eps}$ and } \mu(X \setminus X_{\delta, \eps}) \leq \eps
\end{equation}
up to projecting $\eta^{\delta, \eps}$ onto $\Delta_{\delta}$ we may also assume that $\eta^{\delta, \eps}$ takes its values in $\Delta_{\delta}$. Hence, defining
\[\beta^{\delta, \eps}(x,y):=\sum_{i=1}^{N_{\delta}} \eta_i^{\delta,\eps}(x) \gamma_i(y), \; \forall (x,y)\in X\times Y,\]
we have $\beta^{\delta, \eps} \in C(X, K_{\delta}) \subset C(X, K)$ and $\beta^{\delta, \eps}$ and $\beta^{\delta}$ coincide on $X_{\delta, \eps}$, so that
\begin{align*}
\int_X c_x(\pi_x) \dd \mu(x) & \leq \int_{X} \int_Y \beta^{\delta}(x,y) \dd \pi_x(y)\dd \mu(x)-\int_{X} c_x^*(\beta^{\delta}_x) \dd \mu(x)+ 2 \delta\\
& = \int_{X} \int_Y \beta^{\delta, \eps}(x,y) \dd \pi_x(y)\dd \mu(x)-\int_{X} c_x^*(\beta^{\delta, \eps}_x) \dd \mu(x)+ 2 \delta\\
&+ \int_{X\setminus X_{\eps, \delta}} \int_Y(\beta^{\delta}(x,y)- \beta^{\delta, \eps}(x,y)) \dd \pi_x(y)\dd \mu(x) + \int_{X\setminus X_{\eps, \delta}}( c_x^*(\beta^{\delta, \eps}_x)- c_x^*(\beta^{\delta}_x) )\dd \mu(x)\\
& \leq   \int_{X\times Y} \beta^{\delta, \eps} \dd \pi -\int_{X} c_x^*(\beta^{\delta, \eps}_x) \dd \mu(x) + 2 \delta +  2 \eps \mathrm{diam}(K)\\
& \leq  \sup_{\beta \in C(X,K)} \left\{\int_{X\times Y} \beta \dd \pi-\int_X c_x^*(\beta_x) \dd \mu(x) \right\} + 2 \delta +  2 \eps \mathrm{diam}(K)
\end{align*}
where we have used \eqref{cstarlip1} and $\mu(X \setminus X_{\delta, \eps}) \leq \eps$ in the fourth line  and the fact that $\beta^{\delta, \eps} \in C(X, K)$  in the last one. We conclude by letting $\eps$ and $\delta$ vanish.  
\smallskip

{\bf Acknowledgments:}  GC acknowledges the support of the Lagrange Mathematics and Computing Research Center.

\bibliographystyle{plain}
\bibliography{references}

\end{document}